\documentclass[9pt,twoside,reqno]{amsart}
\allowdisplaybreaks
\usepackage{amsmath,amstext,amssymb,epsfig}
\usepackage{graphicx}
\usepackage{mathrsfs}
\calclayout
\makeatletter
\renewcommand{\@seccntformat}[1]{\bf\csname the#1\endcsname.}
\renewcommand{\section}{\@startsection{section}{1}
	\z@{.7\linespacing\@plus\linespacing}{.5\linespacing}
	{\normalfont\upshape\bfseries\centering}}
\renewcommand{\@biblabel}[1]{\@ifnotempty{#1}{#1.}}
\makeatother
\theoremstyle{plain}
\newtheorem{thm}{Theorem}[section]

\newtheorem{prop}[thm]{Proposition}
\newtheorem{cor}[thm]{Corollary}

\theoremstyle{definition}
\newtheorem{ex}[thm]{Example}
\newtheorem{defn}[thm]{Definition}
\newtheorem{rem}{Remark}[section]

\usepackage[parfill]{parskip}
\usepackage[german]{varioref}
\usepackage[all]{xy}
\usepackage{amsmath}
\usepackage{stmaryrd}
\usepackage{color}
\def\M{{\mathcal M}}
\def\R{{\mathcal R}}
\def\L{{\mathcal L}}
\def\N{{\mathcal N}}

\def\T{{\mathcal T}}
\def\S{{\mathcal S}}
\def\>{\succ}
\def\<{\prec}

\def\M{{\mathcal M}}

\def\b{\beta}
\def\a{\alpha}
\def\l{\lambda}
\def\p{\partial}

\def\m{\mu}

\begin{document}
\title[Sania Asif\textsuperscript{1,2}, Yao Wang* \textsuperscript{1,2},  Bouzid Mosbahi\textsuperscript{3}, Imed Basdouri \textsuperscript{4}
]{Cohomology and deformation theory of $\mathcal{O}$-operators on Hom-Lie conformal algebras}
\author{ Sania Asif\textsuperscript{1, 2}, 
Yao Wang \textsuperscript{1, 2}, Bouzid Mosbahi \textsuperscript{3}
, Imed Basdouri\textsuperscript{4}}
  \address{\textsuperscript{1,2}  School of Mathematics and Statistics, Nanjing University of Information Science and Technology,
   		Nanjing, Jiangsu, 210044, P. R. China.}
   	\address{\textsuperscript{1,2}Center for Applied Mathematics of Jiangsu Province/Jiangsu International Joint Laboratory on System Modeling and Data Analysis, Nanjing University of Information Science and Technology, Nanjing, Jiangsu, 210044, P. R. China.}
 \address{\textsuperscript{3}  University of Sfax, Faculty of Sciences of Sfax, BP 1171, 3038, Sfax, Tunisia.}
 \address{\textsuperscript{4}University of Gafsa, Faculty of Sciences of Gafsa, 2112 Gafsa, Tunisia.}
\email{\textsuperscript{1*}11835037@zju.edu.cn}
\email{\textsuperscript{2*}wangyao@nuist.edu.cn}
	\email{\textsuperscript{3}bouzidmosbahy@gmail.com}
\email{\textsuperscript{4}basdourimed@yahoo.fr}

		\date{\today}
		\thanks{This work is supported by the Jiangsu Natural Science Foundation Project (Natural Science Foundation of Jiangsu Province), Relative Gorenstein cotorsion Homology Theory and Its Applications (No.BK20181406).}
			\begin{abstract}In the present paper, we aim to introduce the cohomology of $\mathcal{O}$-operators defined on the Hom-Lie conformal algebra concerning the given representation. To obtain the desired results, we describe three different cochain complexes and discuss the interrelation of their coboundary operators. And show that differential maps on the graded Lie algebra can also be defined by using the Maurer-Cartan element. We further find out that, the $\mathcal{O}$-operator on the given Hom-Lie conformal algebra serves as a Maurer-Cartan element and it leads to acquiring the notion of a differential map in terms of $\mathcal{O}$-operator $\delta_\T$. Next, we provide the notion of Hom-pre-Lie conformal algebra, that induces a sub-adjacent Hom-Lie conformal algebra structure. The differential $\delta_{\b,\a}$ of this sub-adjacent Hom-Lie conformal algebra is related to the differential $\delta_\T$. Finally, we provide the deformation theory of $\mathcal{O}$-operators on the Hom-Lie conformal algebras as an application to the cohomology theory, where we discuss linear and formal deformations in detail.
		\end{abstract}
		\footnote{The second author is corresponding author. \textbf{Emails:} 11835037@zju.edu.cn; wangyao@nuist.edu.cn.}
		\maketitle
		\section{ Introduction}Hom-Lie conformal algebra represents a crucial category within conformal algebras, wherein Lie conformal algebras are twisted by the $\mathbb{C}$-linear endomorphism maps, under a specific set of axioms. Yuan et al. pioneered the study of these algebras in \cite{Yuan}, introducing the key concepts of Hom-Lie conformal algebras, Gel'fand-Dorfman bialgebras and established an equivalence between them. Subsequent advancements extended this study to the case of Hom-Lie conformal superalgebras in \cite{Yuan2}. In the study carried out by Zhao et al. in \cite{Zhao-Yuan-Chen}, an in-depth analysis of the cohomology and deformation of Hom-Lie conformal algebras was presented, along with the significant insights into the properties of derivations on the multiplicative Hom-Lie conformal algebras. Guo et al. explored the corresponding concepts in the context of the Hom-Lie conformal superalgebras, see \cite{Guo-Dong-Wang}. The growing significance of this subject prompted the introduction of a Hom-pre Lie conformal bialgebra structure in \cite{Guo-Zhang-Wang}. Successively, numerous researchers have explored this concept, and provide valuable results, particularly in the domain of $\delta$-Hom Jordan Lie conformal algebra, as evidenced in \cite{Guo-Wang}.  
		\par The $\mathcal{O}$-operator, a mathematical structure often referred to as a relative Rota-Baxter operator, holds a significant importance. It serves as a generalization of the Rota-Baxter operator with weight $0$ in the presence of bimodule structure. Its study spans a spectrum of algebraic structures, including associative algebra, Lie algebra, Hom-Lie algebra, Malcev algebra, $L_{\infty}$-algebra, Poisson algebra, and certain conformal algebras. Initially introduced by Bai et al. in the study of classical Yang-Baxter equations \cite{Bai}, the $\mathcal{O}$-operator has since gained considerable interest, giving rise to numerous research endeavors. Yuan, in \cite{Yuan}, explored its connections with Loday algebras and conformal $NS$-algebras. This exploration extended to the cohomology of $\mathcal{O}$-operators within the domain of associative conformal algebras. In the work by Asif et al. \cite{Asif-Wang-Wu}, the study scrutinized the relationship between $\mathcal{O}$-operators, Rota-Baxter operators, and Nijenhuis operators in the frame work of an Hom-associative conformal algebras. The non-commutative associative conformal algebra, under the $\lambda$-bracket, gives rise to the Lie conformal algebra. Yuan and Liu undertaken a thorough investigation into the cohomology and deformations of Rota-Baxter operators on the Lie conformal algebras, constructing $L_\infty$-algebras from (quasi)-Twilled Lie conformal algebras, as detailed in \cite{Yuan-Liu}. Additional investigation of the $\mathcal{O}$-operators was conducted by Feng et al. in \cite{Feng-Chen}, with a focus on Leibniz conformal algebras. Their study considered conformal representations and provided a thorough investigation of the deformation theory via cohomology, providing a  Maurer-Cartan characterization in terms of $\mathcal{O}$-operators. The exploration into the cohomology of $\mathcal{O}$-operators extended to Hom-associative algebras was conducted by Chtioui et al. in \cite{Chtioui-Mabrouk-Makhlouf}. The same authors further widen this subject in the context of Lie triple systems, as detailed in \cite{Chtioui-Hajjaji-Mabrouk-Makhlouf}. For the even deeper understanding of $\mathcal{O}$-operators, reader is referred to \cite{Mishra-Naolekar, Tang-Guo-Sheng, Uchino}. 
	\par Due to the significance of Hom-Lie conformal algebras and $\mathcal{O}$-operators discussed in the preceding paragraphs, we aim to explore the $\mathcal{O}$-operator within the framework of the Hom-Lie conformal algebra. Taking idea from \cite{Asif-Wang-Yuan,Tang-Guo-Sheng,Yuan-Liu}, our objective includes the detailed examination of the cohomology and deformation theory of the operator algebras under consideration. Certainly, our literature survey features a notable absence of studies addressing the cohomology of Hom-associative conformal algebras and Lie conformal algebras. Therefore, our study about the Hom-Lie conformal algebras is anticipating to extend naturally to the case of Lie conformal algebras by considering twist map to be identity. To achieve our objective, we describe three distinct cochain complexes, examining the interrelation of their coboundary operators. We descibe that  differential map on the graded Lie algebra can be defined by using the Maurer-Cartan element. Additionally, we determine that the $\mathcal{O}$-operators on the Hom-Lie conformal algebra acts as a Maurer-Cartan element, leading to the notion of a differential map formulated in terms of the $\mathcal{O}$-operator, denoted by $\delta_\T$. We present the idea of a Hom-pre-Lie conformal algebra, inducing a sub-adjacent Hom-Lie conformal algebra structure. Notably, we observe that its differential, denoted as $\delta_{\b,\a}$, correlates to the differential $\delta_\T$.
	Finally, we describe the deformation theory of the $\mathcal{O}$-operators on the Hom-Lie conformal algebras, exhibiting as an application to the  cohomology theory. Our detailed analysis encompasses both linear and formal deformations, where we discuss the linear and formal deformations in details.  
	\par The paper is coherently structured as follows: In Section 2, we provide a thorough examination of key concepts associated to Hom-Lie conformal algebra and its representation. Moving forward to Section 3, we revisit the fundamental notion and introduce the cohomology complex $$(C^*_{\a,\b}(\L, \M), \delta_{\a,\b})$$ characterizing a Hom-Lie conformal algebra with respect to its representation. Furthermore, by considering the cochain $C^*_{\a,\a}(\L,\L)$, we define the $NR$-bracket $[\cdot,\cdot]_{NR}^{\a}$ and introduce the coboundary operator employing the Maurer-Cartan element $m_c$. Our investigation demonstrates a close connection between the "differential" of the Hom-Lie conformal algebra and the "coboundary operator in terms of the Maurer-Cartan element" of the Hom-Lie conformal algebra. This analysis ends up in the formulation of the cohomology complex $$(C^*_{\a,\a}(\L, \L), \delta_{\a}).$$ Advancement continues with the introduction of the graded Lie algebra for the graded space of $n$-cochains of the semi-direct product of $\L$ and $\M$. We showcase its Maurer-Cartan element as the sum of $m_c$ and $\rho$. In Section 4, we explicitly obtain the differential for the graded Lie algebra, constructed in the preceding section. Moreover, we define the concept of $\mathcal{O}$-operators within the framework the Hom-Lie conformal algebra. by employing the structure of graded Lie algebra on the deformation complex of a Hom-Lie conformal algebra, we establish a graded Lie algebra whose Maurer-Cartan elements precisely correspond to the $\mathcal{O}$-operators on Hom-Lie conformal algebras. This leads to the present the idea of a differential graded Lie algebra associated with an $\mathcal{O}$-operator. The final Section 5 is devoted to the discussion of the deformation of $\mathcal{O}$-operators on the Hom-Lie conformal algebra. We commence with an exploration of linear deformation and further delve into the exploration of formal deformation. This paper's structure presents the comprehensive understanding of the  systematic investigation of Hom-Lie conformal algebra and $\mathcal{O}$-operators, focusing their cohomological and deformation aspects.
		\par Through out the paper, all the vector spaces, tensor products and (bi-)linear maps are considered over a field of complex numbers $\mathbb{C}$.
		
		\section{ Preliminaries}\begin{defn}
			A Hom-Lie conformal algebra $\L$ is a $\mathbb C[\p]$-module endowed with a $\mathbb C$-bilinear map
			$[\cdot_\l\cdot]:\L \otimes \L \to \L[\l ], x\otimes y
			\mapsto [x_\l y]$,
			called the $\l $-bracket, and a $\mathbb{C}$-linear map $\a:\L\to\L$ such that $\a\p=\p\a$, satisfying the following axioms :
			\begin{align*}[\p (x)_\l y] = -\l [x_{\l} y], [x_\l \p y] &= (\p +\l )[x_{\l} y], &\textit{(conformal-sesquilinearity)} \\
			[x_\l y]&= -[y_{-\l-\p} x], &\textit{(conformal skew-symmetry)} \\
			[\a(x)_\l [y_\m z]] &= [[x_\l y]_{\l+\m}\a(z)]+[\a(y)_{\m} [x_\l z]], &\textit{(conformal Hom-Jacobi identity)}
			\end{align*}
			for all $x,y,z \in \L$ and $\l,\m\in \mathbb{C}$.
		\end{defn}
		Let us consider two $\mathbb{C}[\p]$-modules $ \M $ and $ \S $. A $\mathbb{C}[\p]$-linear map $g_{\l}:\M\to \mathbb{C}[\p]\otimes \S$ is called a conformal linear map if it satisfies the following axioms:
		\begin{equation}g_{\l}(\p m)= (\p+\l) {g}_{\l}m,\quad \forall ~m\in \M.	
		\end{equation} Any vector space containg all conformal linear maps is denoted by $CHom(\M,\S)$. If we assume that $\M=\S$, then the underlying space is denoted by $CHom(\M,\M)$, which is nothing but the endomorphism $Cend(\M)$ that satisfies the following equation : $$(\p g)_\l=-\l g_\l,\quad \forall ~g_{\l}\in Cend(\M) $$ and becomes as a $\mathbb{C}[\p]$-module.  Note that a $\mathbb{C}[\p]$-module $Cend(\M)$ carries the structure of a Lie conformal algebra on $\M$ under the $\l$-bracket given by \begin{equation}
		{[f_\l g]}_\m m= f_\l (g_{\m-\l} m) - g_{\m-\l} (f_{\l} m) \quad \forall ~f,g \in Cend(\M),~m\in \M~ and~ \l,\m\in \mathbb{C}. 
		\end{equation} This Lie conformal algebra is termed as a general Lie conformal algebra, we denote it by $gc(\M)$.
		
		\begin{defn}
			A conformal module $\M$ over a Lie conformal algebra $\L$ is a $\mathbb C[\p]$-module endowed with a $\mathbb C$-bilinear map $\L \otimes \M \to \M[\l]$ defined by  $(x,m) \mapsto x_{\l} m,$ subject to the following relations:\begin{equation}
			\begin{aligned}
			(\p x)_\l m &= -\l (x_\l m),\quad x_\l (\p m) = (\p +\l )(x_\l m),\\
			([x_\l y])_{\l+\m} m &= x_\l (y_\m m)- y_\m (x_\l m),
			\end{aligned}
			\end{equation}
			for all $x,y\in \L$ and $m \in \M$, $\l,\m\in \mathbb{C}$.\end{defn}
	 	A conformal module $\M$ over a Lie conformal algebra $\L$ is called finite if $\M$ is finitely generated
		over $\mathbb C[\p]$. It is easy to see that a conformal module $\M$ over a Lie conformal algebra $\L$ is the same as a homomorphism of Lie conformal algebras $\rho : \L\to gc(\M)$, which is called a representation of $\L$ in the $\mathbb C[\p]$-module $\M$.

	\begin{defn}A representation of a Hom-Lie conformal algebra $(\L, [\cdot_\l\cdot], \alpha)$ is a triple $(\rho, \M, \beta)$ such that $\rho: \L \to gc(\M)$ is a $ \mathbb{C}[\p] $-liner map and $(\beta,\M)$ is $\mathbb C[\p]$-module satisfying the following properties:\begin{align*}
		 \rho \partial=& \partial \rho, \\ \rho(\partial(x))_\lambda =& -\lambda \rho(x)_\lambda,\\ [\rho(x)_\lambda \rho(y)]_{\l+\m} =& \rho([x_{\lambda}y])_{\lambda+\mu},\\ \rho([x_{\lambda}y])_{\lambda+\mu}\beta =& \rho(\alpha(x))_{\lambda}\rho(y)_\mu - \rho(\alpha(y))_{\mu}\rho(x)_{\lambda}.\end{align*}
		Sometimes in  this paper, for simplification we use the notation
			$\{x, m\}:= \rho(x)_{\lambda}(m),$ for all $x \in \L, m\in \M$ and $\l\in \mathbb{C}$.\end{defn}
		\begin{ex}\label{ex1}For any integer $p\geq 0$, we can define the $p$-adjoint representation of a Hom-Lie conformal algebra $(\L, [ \cdot_\lambda\cdot], \alpha)$ on $\L$ as follows\begin{eqnarray}ad(x)^p_{\lambda}(y):= [\alpha^{p}(x)_\lambda y].\end{eqnarray} For simplicity, we can denote it by $( ad^p,\L, \alpha).$ For more detail see \cite{Yuan}.\end{ex}
		Define the conformal dual of a $\mathbb C[\p]$-module $\M$ as $\M^{*c} = CHom(M, \mathbb{C})$, where $\mathbb{C}$ is viewed
		as the trivial $\mathbb C[\p]$-module, that is $\M^{*c} = \{a : \M \to \mathbb C[\l] ~|~ a ~is~ \mathbb{C}-linear ~and ~a_\l(\p b) =\l (a_\l b)\}.$ Just like representation of a conformal module, we can define the dual representation of the dual conformal module by the triple $(\rho^*, \M^*, \b^{-1*})$ consisting of a map $\rho^*:\L\to gc(\M^*)$, such that
		 \begin{align*}
		 (\rho^*(a)_\l\varphi)_\m(m):=(\rho^*(\a^{-1}(a))_\l(\b^{-2})^*\varphi)_\m(m)=- \varphi_{\m-\l}(\rho( \a^{-1}(a))_\l \b^{-2}(m)),
		 \end{align*} for $
		 a\in \L,\varphi\in \M^*,m\in \M ~and~ \l,\m\in \mathbb C.$
	
		In particular, we can define the coadjoint representation of the Hom-Lie conformal algebra by the triple $(ad^{*}, \L^*, \alpha^{-1*})$, such that  \begin{eqnarray}(ad^{*}(x) _{\l }\varphi)_\m(y):= - \varphi_{\m-\l}([\a^{-1}(x)_{\l} \b^{-2}(y)]),~\forall~ x, y\in \L,~\varphi \in \L^*.\end{eqnarray}
		\begin{ex} Given a Lie conformal algebra $\L$ with a Lie conformal algebra homomorphism $\a : \L \to \L$, we can define a Hom-Lie conformal algebra as the triplet $(\L, \a ([\cdot_\l\cdot]), \a)$, where $[\cdot_\l \cdot]$ is the underlying Lie conformal bracket or $\l$-bracket.
		\end{ex}
	\begin{prop}\label{directsum}
			Let $(\L, [\cdot _\lambda\cdot], \alpha)$ be a Hom-Lie conformal algebra and $(\rho, \M, \beta)$ is a representation of the Hom-Lie conformal algebra $\L$. The vector space $\L\oplus \M$ is a $\mathbb C[\p]$-module by $ \p(a,m)=(\p_\L a,\p_\M m) $. In the presence of $\lambda$-bracket $[\cdot_\lambda\cdot]: (\L\oplus \M)\otimes (\L\oplus \M)\to (\L\oplus \M)[\l]$, defined~by:
			\begin{align*}[(x + m)_{\lambda}(y + n)]_{\rho}=[x_\lambda y]+ \rho(x)_\lambda n - \rho(y)_{-\partial- \lambda}m,
			\end{align*} and
		a twist map $\alpha+ \beta :\L\oplus \M\to \L\oplus \M$ defined~by: \begin{align*}
		(\alpha + \beta)(x+ m)= \alpha(x)+ \beta(m),
		\end{align*} the vector space $\L \oplus \M$ carries the structure of a Hom-Lie conformal algebra. 
		\end{prop}
		\begin{proof}The proof of proposition is straight forward (see \cite{Yuan} for more details).
		\end{proof}
		The Hom-Lie conformal algebra $(\L \oplus \M, [\cdot_{\lambda}\cdot]_{\rho}, \alpha+\beta)$ is called the semi-direct product Hom-Lie conformal algebra, we denote it by $\L\ltimes\M$.
		\section{ Cochain complexs on the Hom-Lie conformal algebras}
This section is further divided into two subsections, where we discuss the cohomology complexes of the Hom-Lie conformal algebras in the framework of representation and Maurer-Cartan elements, respectively.
	\subsection{ The cohomology complex for a Hom-Lie conformal algebra with representation:}	
	Let $(\L, [\cdot_\l\cdot],\a)$ be a Hom-Lie conformal algebra with a representation $(\rho, \M, \b)$ on a $\mathbb {C}[\p]$-module $\M.$ Assume that $C^p_{\a,\b}(\L, \M)$ is a space of $\mathbb C$-linear maps $$f_{\l_1, \l_2, \cdots, \l_{p-1}}: \L^{\otimes p}\to \mathbb{C}[\l_1,\l_2,\cdots,\l_{p-1}]\otimes \M,$$ defined by $$x_1\otimes x_2\otimes\cdots\otimes x_p \mapsto f_{\l_1,\l_2,\cdots,\l_{p-1}}(x_1,x_2,\cdots,x_p).$$ This map satisfies 
	\begin{equation}\label{eq9}
	\b(f_{\l_1,\cdots,\l_{p-1}}(x_1, x_2, \cdots, x_p))= f_{\l_1, \cdots, \l_{p-1}}(\a(x_1), \a(x_2), \cdots, \a(x_p)), 
	\end{equation} the following conditions of conformal sesqui-linearity:
	\begin{equation}\label{eq10}
	\begin{aligned}
	f_{\l_1, \l_2,\cdots, \l_{p-1}}&(x_1,x_2,\cdots,\p(x_i),\cdots,x_p)\\&=\begin{cases}
	-\l_{i} f_{\l_1,\l_2,\cdots,\l_{p-1}}(x_1, x_2, \cdots , x_p),&i= 1, \cdots, p-1, \\
	 (\p+\l_1+\l_2+\cdots+\l_{p-1}) f_{\l_1,\l_2,\cdots,\l_{p-1}}(x_1, x_2, \cdots , x_p),&i=p. \end{cases}
	\end{aligned}\end{equation}
	and conformal skew-symmetry:
	\begin{equation}\label{eq11}
	\begin{aligned}
	f_{\l_1,\l_2,\cdots,\l_{p-1}}(x_1,x_2,\cdots,x_p)=(-1)^\tau f_{\l_{\tau(1)},\l_{\tau(2)},\cdots,\l_{\tau(p-1)}}(x_{\tau(1)}, x_{\tau(2)}, \cdots , x_{\tau(p)})|_{\l_{p}\mapsto \l_{p}^\dagger},
	\end{aligned}\end{equation}
	where $\l_{p}^\dagger=-\sum_{i=1}^{p-1}\l_{i}-\p^{\M}$, for all $x_1, x_2, \cdots, x_p\in \L$. Moreover, above equation shows that the $ \mathbb{C}$-linear map $f$ is skew-symmetric with respect to simultaneous permutation of $ x_{j} $'s and $\l_j$'s.
	
The coboundary map $\delta_{\a,\b} : C^p_{\a,\b}(\L, \M ) \to C^{p+1}_{\a,\b}(\L,\M)$, defined on the space of $p$-cochains $C^p_{\a,\b}(\L, \M ) $ is given by : 
\begin{equation}\label{coboundarymap}\begin{aligned}&(\delta_{\a,\b} f)_{\l_1,\cdots,\l_{p}}(x_1,\cdots, x_{p+1})\\=& \sum_{i=1}^{p+1} (-1)^{i+1}\rho(\a^{p-1}(x_i))_{\l_i}f_{\l_1,\cdots,\hat{\l_{i}},\cdots, \l_{p}}(x_1,\cdots,\hat{x_i},\cdots, x_{p+1})\\&+ \sum_{i<j}(-1)^{i+j} f_{\l_i+ \l_j ,\l_1,\cdots, \hat{\l_{i}}, \cdots, \hat{\l_{j}}, \cdots, \l_{p}}([{x_{i}}_{\l_i} x_j], \a(x_1), \cdots, \hat{x_i},\cdots, \hat{x_j},\cdots,\a(x_{p+1})). \end{aligned}\end{equation} 
In the above equation $f\in C^{p}(\L, \M)$, $x_1, x_2, \cdots, x_p \in \L$, and $\hat{x_i}$ shows the omission of the entry $x_{i}$ (see \cite{Zhao-Yuan-Chen} for more detail).\\
 Consider the space $C^*_{\a,\b}(\L,\M):=\bigoplus_{p\geq1}C^p_{\a,\b}(\L, \M)$ is the graded space of $p$-cochains. Assuming $f\in C^p(\L, \M)$, $\delta_{\a,\b} (f)\in C^{p+1}(\L,\M)$ and $ \delta^2_{\a,\b}(f)=0$. In this case $(C^*_{\a,\b}(\L, \M), \delta_{\a,\b})$ is a cochain complex for the Hom-Lie conformal algebra $(\L, [\cdot_\l\cdot], \a)$ with coefficients in the representation $(\rho,\M, \b)$. We denote the cohomology space associated to the cochain complex $(C^*_{\a,\b}(\L, \M), \delta_{\a,\b})$ by $\mathcal H^*_{\a,\b}(\L, \M)$.
	\begin{rem}\label{rem2.1}
	Let $(\L, [\cdot_\l\cdot],\a)$ be a regular Hom-Lie conformal algebra and $(\rho,\M,\b)$ be a representation of the Hom-Lie conformal algebra, where $\b: \M \to \M$ is invertible endomorphism. Then $0$-cochain is defined by $C^0_{\a,\b}(\L,\M):= \{m\in \M| \b=id\},$ and the coboundary map $\delta_{\a,\b}: C^0_{\a,\b}(\L,\M)\to C^1_{\a,\b}(\L,\M)$ on $0$-cochain is given by $$ \delta_{\a,\b}(m)(x) := \rho(\a^{-1}(x))_\l(m), $$ for all $m \in C^0_{\a,\b}$, $x \in \L$ and $\l\in \mathbb{C}$.
\end{rem}	
\subsection{ The cohomology complex for a Hom-Lie conformal algebra in terms of Maure-Cartan element:} Observe that, if $\M = \L,$ $\a = \b$ and the action $\rho : \L\otimes \M \to \M[\l]$ is given by the underlying Hom-Lie conformal bracket. In this case, the coboundary map in Eq. \eqref{coboundarymap} is turned to be \begin{equation}\label{coboundarymap2}\begin{aligned}&(\delta_{\a} f)_{\l_1,\cdots,\l_{p}}(x_1,\cdots, x_{p+1})\\=& \sum_{i=1}^{p+1} (-1)^{i+1} [ \a^{p-1}(x_i) _{\l_i}f_{\l_1,\cdots,\hat{\l_{i}},\cdots, \l_{p}}(x_1,\cdots,\hat{x_i},\cdots, x_{p+1})]\\&+ \sum_{i<j}(-1)^{i+j} f_{\l_i+ \l_j ,\l_1,\cdots, \hat{\l_{i}}, \cdots, \hat{\l_{j}}, \cdots, \l_{p}}([{x_{i}}_{\l_i} x_j], \a(x_1), \cdots, \hat{x_i},\cdots, \hat{x_j},\cdots,\a(x_{p+1})). \end{aligned}\end{equation}The cochain complex thus obtained is  denoted by  $(C^*_{\a}(\L,\L),\delta_{\a}).$ And this complex is the same as the cochain complex of the Hom-Lie conformal algebra $(\L, [\cdot_\l\cdot],\a)$ (defined in \cite{Yuan2}). The cohomology of the complex $(C^*_\a(\L, \L), \delta_{\a})$ serves as the cohomology of Hom-Lie conformal algebra and is denoted by  $\mathcal H^*_{\a}(\L, \L)$.
		
		\subsubsection{ Hom-Lie conformal algebras and their representations in terms of Maurer-Cartan elements:}
		Let us now consider the graded vector space $C^*_\a(\L,\L)=\bigoplus_{n\in \mathbb{Z}}C^n_\a(\L,\L)$, where $C^n_\a(\L,\L)$ is the
		space of $\mathbb{C}$-linear maps satisfying Eqs. \eqref{eq9}-\eqref{eq11}. From \cite{Asif-Wang-Yuan}, we recall that for any $f\in C^{m}_\a (\L, \L)$ and $g \in C^{n}_\a (\L, \L)$, a circle product $f\circledcirc g: C^{m}_\a (\L, \L)\otimes C^{n}_\a (\L, \L)\to C^{m+n-1}_\a (\L, \L)$ is defined by the following expression:
		\begin{equation}\label{circleproduct}
		\begin{aligned}
		(f\circledcirc g)&_{\l_1,\cdots,\l_{m+n-2}}(x_1,\cdots,x_{m+n-1})=\sum_{\tau\in S_{n,m-1}} (-1)^{\tau} f_{\l_{\tau(1)}+\cdots+\l_{\tau({n})},\l_{\tau({n+1})},\cdots,\l_{\tau({m+n-2})}}\\& (g_{\l_{\tau(1)},\cdots,\l_{\tau({n-1})}}(x_{\tau(1)},\cdots, x_{\tau(n)}), \a^{n-1}(x_{\tau(n+1)}),\cdots, \a^{n-1}(x_{\tau(m+n-1)})).
		\end{aligned}
		\end{equation} 
		for $x_1,\cdots ,x_{m+n-1}\in \L$. Whereas $S_{n,m-1}$ denotes $\{1,2,\cdots,m+n-1\}$-shuffle in the permutation group of $S_{n,m-1}$. The notation $|\tau|$ denotes the permutation $\tau$, for arbitrary $\tau\in S_{m+n-1}$.
	 \begin{rem} If $m=n=2$  in the above equation, we get 
	\begin{equation}\label{eqcirc22}
\begin{aligned}
&(f\circledcirc g)_{\l_1,\l_2}(x_1,x_2,x_3) =\\
&f_{\l_1 + \l_2}(g_{\l_1}(x_1, x_2), \alpha(x_3)) - f_{\l_1 + \l_2}(g_{\l_2}(x_1, x_3), \alpha(x_2)) + f_{\l_1 + \l_2}(g_{\l_1}(x_2, x_3), \alpha(x_1)).
\end{aligned}
\end{equation}
\end{rem}
		Now by using Eq. \eqref{circleproduct}, we can define a Nijenhuis-Richardson bracket (also called as $NR$-bracket) of degree $-1$ on graded vector space $C^*_\a(\L, \L)$ as follows: $$[f,g]_{NR}^{\a}:=f \circledcirc g -(-1)^{(m-1)(n-1)}g\circledcirc f.$$
		This $NR$-bracket is an element of the space of $C_\a^{m+n-1}(\L,\L) $. The graded vector space $ C_\a^{*}(\L, \L) $ along with the Nijenhuis- Richardson bracket $[\cdot, \cdot]^\a_{NR}$ forms a graded Lie algebra $(C_\a^{*}(\L, \L), [\cdot, \cdot]^\a_{NR})$.
	In particular, if we have an  element of a $2$-cochain of a graded Lie algebra, i.e., $m_c \in C^{2}_\a(\L, \L)$ satisfying $[m_c, m_c]^{\a}_{NR}= 0$, we can get a Hom-Lie conformal algebra by defining:
	\begin{align}
	{m_c}_{\l}(x, y):= [x_\l y],~\forall ~x, y\in \L.
	\end{align}For the related concept on the Lie conformal algebras, readers are refer to \cite{Yuan-Liu}. The element $m_c$ is called as Maure-Cartan element of the graded Lie algebra. As we are familiar with the concept that, any graded Lie algebra in the presence of differential operator provides a differential graded Lie algebra denoted by ${dgLa}$. With the help of Maurer-Cartan element, we can define the differential operator on the graded Lie algebra $(C_\a^{* }(\L, \L),[\cdot,\cdot]^\a_{NR})$ as follows:
	 \begin{prop} Consider a Hom-Lie conformal algebra $(\L, [\cdot_\lambda \cdot], \alpha)$. The differential operator $d_{m_c}:C_\a^{*}(\L, \L)\to C_\a^{*+1 }(\L, \L)$ is a $\mathbb{C}$-linear map, given by $d_{m_c}= [m_c ,-]^\a_{NR}$. 
\end{prop}   
\begin{proof}Consider that for any $f\in C_\a^{n}(\L, \L)$ and $m_c\in C_\a^{2}(\L, \L)$, we have \begin{align*}
	d_{m_c}(f)=& [m_c ,f]^\a_{NR}\\=& (m_c \circledcirc f -(-1)^{(1)(n-1)}f\circledcirc m_c) (x_1,\cdots,x_{n+1})
	\\=& m_c \circledcirc f  (x_1,\cdots,x_{n+1})-(-1)^{n}f\circledcirc m_c (x_1,\cdots,x_{n+1})
	\\=& (m_c \circledcirc f)_{\l_1,\l_2,\cdots,\l_n}  (x_1,\cdots,x_{n+1})-(-1)^{n}(f\circledcirc m_c)_{\l_1,\l_2,\cdots,\l_n} (x_1,\cdots,x_{n+1})
	\\=& \sum_{\tau\in S_{n,1}} (-1)^{\tau} {m_c}_{\l_{\tau(1)}+\cdots+\l_{\tau({n})}} (f_{\l_{\tau(1)},\cdots,\l_{\tau({n-1})}}(x_{\tau(1)},\cdots, x_{\tau(n)}), \a^{n-1}(x_{\tau(n+1)})) \\&
	-(-1)^{n} \sum_{\tau\in S_{2,n-1 }} (-1)^{\tau} f_{\l_{\tau(1)}+\l_{\tau({2})},\l_{\tau({3})},\cdots,\l_{\tau({n})}} ({m_c}_{\l_{\tau(1)}}(x_{\tau(1)}, x_{\tau(2)}), \a (x_{\tau(3)}),\cdots, \a (x_{\tau(n+1)})) 
	\\=& \sum_{\tau\in S_{n,1}} (-1)^{\tau}  [{f_{\l_{\tau(1)},\cdots,\l_{\tau({n-1})}}(x_{\tau(1)},\cdots, x_{\tau(n)})}_{\l_{\tau(1)}+\cdots+\l_{\tau({n})}} \a^{n-1}(x_{\tau(n+1)})] \\&
	-(-1)^{n} \sum_{\tau\in S_{2,n-1}} (-1)^{\tau} f_{\l_{\tau(1)}+\l_{\tau({2})},\l_{\tau({3})},\cdots,\l_{\tau({n})}} ([{x_{\tau(1)}}_{\l_{\tau(1)}} x_{\tau(2)}], \a (x_{\tau(3)}),\cdots, \a (x_{\tau(n+1)})).
	\end{align*}
\end{proof}	
 
 This differential operator $d_{m_c}$ coincides with the coboundary operator $\delta_{\a,\b}$ given by Eq. \eqref{coboundarymap}. The relationship between them is given by 
 \begin{align*}
 d_{m_c}( f)= (-1)^{1+n}\delta_{\a,\b} (f).
 \end{align*}Further if $f\in C_\a^{n}(\L, \L)$ and  $g\in C_\a^{l}(\L, \L)$, we can define differential of the $NR $-bracket $[f,g]_{NR}$ by the following way: \begin{equation*}
			\begin{aligned}
				d_{m_c}([f,g]_{NR})=& (-1)^{n+l}[m_c,[f,g]^\a_{NR}]^\a_{NR}\\=& (-1)^{n+l} ([[m_c,f]^\a_{NR},g]^\a_{NR}+ (-1)^{n+1}[f,[m_c,g]^\a_{NR}]^\a_{NR})\\=&
			((-1)^{n+l}[d_{m_c} (f),g]^\a_{NR}+ (-1)^{l+1}[f,d_{m_c}(g)]^\a_{NR})
			\\=&
			((-1)^{l-1}[\delta_{\a,\b} (f),g]^\a_{NR}+ [f,\delta_{\a,\b}(g)]^\a_{NR}).
			\end{aligned}
				\end{equation*}
			So, we obtain \begin{equation*}
			((-1)^{l-1}[\delta_{\a,\b} (f),g]^\a_{NR}+ [f,\delta_{\a,\b}(g)]^\a_{NR})=(-1)^{n+l}\delta _{\a,\b} [f,g]_{NR}^\a.
			\end{equation*} 
			Further if $g\in C^0_\a(\L,\L)$, we get the following expressions:
			 \begin{equation*}
			(-[\delta_{\a,\b} (f),g]^\a_{NR}+ [f,\delta_{\a,\b}(g)]^\a_{NR})=(-1)^{n}\delta _{\a,\b} [f,g]_{NR}^\a.
			\end{equation*}
	\subsection{ The cohomology complex for a semi-direct product Hom-Lie conformal algebra:}		
	By the proposition \eqref{directsum}, we know that the triple $(\L \oplus \M, [\cdot_\l\cdot],\a+\b)$ forms a Hom-Lie conformal algebra. Now we consider the corresponding graded Lie algebra \begin{equation}\label{eqgraded}
\L^*:= (C^{*}_{\a,\b}(\L \oplus \M, \L \oplus \M), {[\cdot_\l\cdot]}^{\a+\b}_{NR}).
\end{equation}We first  recall from \cite{Uchino} that  
 \begin{align*}
 C^{n}_{\a,\b}(\L \oplus \M)=\sum_{n=k+l}C^{n}_{\a,\b}(\L^{l,k} , \L) \oplus  \sum_{n=k+l}C^{n}_{\a,\b}(\L^{l,k}, \M),
 \end{align*} the term $\L^{l,k}$ is the direct sum of  $ l+k $-tensor powers of $ \L $ and $ \M $, where $ l $ and $ k  $ are the numbers of  $ \L $  and $\M$. For example $\L^{2,1}= \M \otimes \L\otimes \L+\L\otimes \M\otimes \L+\L\otimes \L\otimes \M$. 
 The tensor space $ (\L \oplus \M)^{\otimes n}= \bigoplus_{n=k+l}\L^{k,l} $. For $n=2$, we have $ (\L \oplus \M)^{\otimes 2}= \L^{2,0}\oplus \L^{1,1}\oplus \L^{0,2}.$ 
 Let $f\in C^{n}_{\a,\b}(\L \oplus \M)$, we say that $f$ has a bidegree $k|l$, if $f\in C^{n}_{\a,\b}(\L^{l,k-1}, \L)$ or $f\in C^{n}_{\a,\b}(\L^{l-1,k}, \M)$, where $n=k+l-1$.
 \par Also assume that, we have three $\mathbb{C}[\p]$-linear maps given by ${m_c}: \L\otimes\L \to\L[\l]$ and $\rho_1,\rho_2 : \L \to gc(\M)$ defined by ${m_c}(x, y)= [x_\l y]$ and $\rho_1(x, a)= \rho(x)_\l a$ and $\rho_2(a, x)= -\rho(x)_{-\p-\l} a$, for all $x, y \in\L$ and $a, b\in \M$. The lift of $m_c$ and $\rho$'s can be defined by $\hat{m_c}((x, a), (y, b))= ({m_c}(x, y), 0)$, $\hat{\rho_1}((x, a), (y, b))= (0, {\rho_1}(x, b))$ and $\hat{\rho_2}((x, a), (y, b))= (0, {\rho_2}(a, y))$. We observe that bidegree of these maps is $1|2$.
 Note that any cochain with the bidegree is termed as homogeneous cochain. By combining these lift maps into a single sum, we get $\hat{\theta}= \hat{m_c}+ \hat{\rho_1}+ \hat{\rho_2}$, which is a homogeneous cochain of bidegree $1|2$. The map $\hat{\theta}$ is not a lift map because there does not exists any map named as $\theta$. Moreover, the multiplication $\hat{\theta}$ is defined by \begin{align*}
 \hat{\theta}((x, a), (y,b))= (\hat{m_c}+\hat{\rho_1} +\hat{\rho_2})((x,a), (y,b)):=& ({m_c} (x,y), {\rho}_1(x,b)+{\rho}_2(a,y))\\=&([x_\l y], {\rho}(x)_\l b- {\rho}(y)_{-\p-\l}a).
 \end{align*}
With the help of above discussion, we have the following proposition.
		\begin{prop}\label{imp}Let the $ (\L,{m_c},\a) $ defines a Hom-Lie conformal algebra and map $(\rho,\M,\b)$ defines its representation if and only if $m_c+ \rho$ is a Maurer-Cartan element of the graded Lie algebra $(\L^*, {[\cdot,\cdot]}^{\a+\b}_{NR})$.
		\end{prop}
	\begin{proof}First, let us observe that the map $m_c+\rho \in \L^2$ if and only if 
		\begin{align*}
		(\a+\b)((m_c+\rho)(x_1+m, x_2+n))=(m_c+\rho)(\a+\b(x_1+m),\a+\b(x_2+n)),
		\end{align*} which is equivalent to the following expressions\begin{align*}
		\a (m_c(x, y))=m_c(\a(x),\a(y)) ~~and~~ \b(\rho(x)_\l(m))=\rho(\a(x))_\l(\b(m)).
		\end{align*}
	Now, we show that the map $m_c+\rho$ is a Maurer-Cartan element if and only if the following Maurer-Cartan equation is satisfied:
	 \begin{equation}\label{eqmc}
		\begin{aligned}
		&([m_c+\rho, m_c+\rho]^{\a+\b}_{NR})_{\l_1,\l_2}(x_1 + m, x_2 + n, x_3 + r)
		\\&= ((m_c+\rho)\circledcirc (m_c+\rho) -(-1)^{1.1} (m_c+\rho)\circledcirc (m_c+\rho))_{\l_1,\l_2}(x_1 + m_1, x_2 + n, x_3 + r) \\&= 2 ((m_c+\rho)\circledcirc (m_c+\rho))_{\l_1, \l_2} (x_1 + m, x_2 + n, x_3 + r)\\& \overset{\eqref{eqcirc22}}{=} 0.
		\end{aligned}
		\end{equation} Since $m_c+\rho\in C^2(\L\oplus \M)$, that is why above $NR$-bracket belongs to the space of $(2+2-1=3)$-cochains. The Eq. \eqref{eqmc} provides us following two expressions: 
		\begin{align*}
		[[{x_1}_\l x_2]_{\l+\m}\a(x_3)]+[[{x_2}_\m x_3]_{\l+\m}\a(x_1)]+[[{x_3}_\l x_1]_{\l+\m}\a(x_2)]=0, 
		\end{align*}
	and	by taking $m=n= 0$, we get
		\begin{align*}
		\rho([{x_1}_{\lambda}x_2])_{\lambda+\mu}\beta(r)= \rho(\alpha(x_1))_{\lambda}\rho(x_2)_\mu(r)- \rho(\alpha(x_2))_{\mu}\rho(x_1)_{\lambda}(r).
		\end{align*}
		for all $x_1, x_2, x_3 \in \L $ and $m, n, r\in \M$. Thus, it is clear that $ (\L, m_c, \a)$ defines a Hom-Lie conformal algebra and $(\rho, \M, \b)$ is a representation of the Hom-Lie conformal algebra.\end{proof}
	\section{ $\mathcal{O}$-operator on Hom-Lie conformal algebra} In this section, we define the notion of $\mathcal{O}$-operators on the Hom-Lie conformal algebra. We use the graded Lie algebra structure on the deformation complex of a Hom-Lie conformal algebra and define a graded Lie algebra whose Maurer-Cartan elements are precisely the $\mathcal{O}$-operators on Hom-Lie conformal algebras. Subsequently, we obtain a differential graded Lie algebra associated to an $\mathcal{O}$-operator.
	\subsection{ $\mathcal{O}$-operator and its representation on the Hom Lie conformal algebra:}
	\begin{defn} Let $(\L, [\cdot_\lambda\cdot], \alpha)$ be a Hom-Lie conformal algebra and $k$ be a non-negative integer. A $\mathbb{C}[\p]$-module homomorphism $\R : \L \to \L$ is called an $p$-Rota Baxter operator of weight $q$ on $(\L, [\cdot_\lambda\cdot],\alpha)$ if $\R\alpha= \alpha \R$ and the following identity is satisfied, for all $x, y\in \L$ \begin{equation}[{\R(x)}_{\lambda} \R(y)]= \R([\alpha^{p}(\R(x))_{\lambda} y]+[x_{\lambda} \alpha^{p}(\R(y))] + q([x_{\lambda}y])).\end{equation} For $\alpha= Id$, above expression coincides with the notion of Rota-Baxter operators on a Lie conformal algebra.\end{defn}
	\begin{defn}\label{def3.2}Let $(\L, [\cdot_\lambda\cdot], \alpha)$ be a Hom-Lie conformal algebra and $(\rho, \M, \beta)$ be a representation of the Hom-Lie conformal algebra. A $\mathbb{C}[\p]$-module homomorphism $\T:\M \to \L$ is called an $\mathcal{O}$-operator on $(\L, [\cdot_\l \cdot],\alpha)$ with respect to the representation $(\rho, \M, \beta)$, if the following conditions hold for $m, n\in \M$:
		\begin{align*}
		\T \beta =&\alpha \T, \\
		[\T (m)_\lambda \T (n)]=&\T(\rho(T(m))_{\lambda}(n)-\rho(\T(n))_{-\partial-\lambda}(m)).\end{align*}
	\end{defn}
	\begin{rem}Let us recall from Example \eqref{ex1}, the $p$-adjoint representation $( ad^p, \L, \alpha)$ of a Hom-Lie conformal algebra $(\L, [\cdot_\lambda\cdot], \alpha)$. Then, a $p$-Rota Baxter operator of weight $0$ on the Hom-Lie conformal algebra $(\L,[\cdot_\lambda \cdot],\alpha)$ is an $\mathcal{O}$-operator on $(\L, [\cdot_\l\cdot], \alpha)$ with respect to the representation $( ad^k, \L, \alpha)$. Thus, the notion of $\mathcal{O}$-operator is a generalization of Rota-Baxter operator and therefore it is also known as relative or generalized Rota-Baxter operator. Another name of $\mathcal{O}$-operator is Kupershmidt operator. \end{rem} 
	\begin{ex}If $\alpha= Id_{\L}$ and $\beta= Id_{\M}$, then the Definition \eqref{def3.2} coincides with the notion of $\mathcal{O}$-operators on the Lie conformal algebra.
	\end{ex}
	\begin{ex}Let $\T : \M \to \L$ is an $\mathcal{O}$-operator on the Lie conformal algebra $(\L, [ \cdot_\lambda\cdot])$ with respect to the representation $(\rho, \M)$. A pair $(\phi_\L, \phi_\M)$ is an endomorphism of the $\mathcal{O}$-operator $\T$ if the following equations hold
		\begin{align*}\T \phi_\M&=\phi_\L \T, \\ \rho(\phi_\L(x))_{\lambda}(\phi_\M (m))&=\phi_\M (\rho(x)_{\lambda}(m)),\end{align*} for all $x\in \L$, $m\in \M$ and $\l\in \mathbb{C}$.\end{ex}
	Let us consider the Hom-Lie conformal algebra $(\L, [\cdot_\lambda\cdot]_{\phi_\L}, \phi_\L)$, where the $\lambda$-bracket is given by $[\cdot_\lambda\cdot]_{\phi_\L}:= \phi_\L([\cdot_\lambda\cdot])$. If we consider the composition $\rho_{\phi_\M}:= \phi_\M\circ\rho,$ then the triple $(\rho_{{\phi}_\M}, \M, \phi_\M)$ is a representation of the Hom-Lie conformal algebra $(\L,[\cdot_\lambda\cdot]_{\phi_\L}, \phi_\L)$. Moreover,
	\begin{align*}
	[\T(m)_\lambda \T(n)]_{\phi_\L}&= \phi_\L([\T(m)_\lambda \T(n)])\\&= \phi_\L(\T(\rho(\T(m))_{\lambda}(n)- \rho(\T(n))_{-\partial-\lambda}(m)))\\&= \T(\phi_\M(\rho(\T(m))_{\lambda}(n) -\rho(\T(n))_{-\partial-\lambda}(m)))\\&= \T(\phi_\M\rho(\T(m))_{\lambda}(n) -\phi_\M\rho(\T(n))_{-\partial-\lambda}(m))\\&= \T(\rho_{\phi_\M}(\T(m))_{\lambda}(n)-\rho_{\phi_\M}(\T(n))_{-\partial-\lambda}(m)),
	\end{align*}
 for all $m, n \in \M$. Thus, the $\mathbb{C}[\p]$-module homomorphism $\T :\M \to \L$ is an $\mathcal{O}$-operator on the Hom-Lie conforml algebra $(\L, [\cdot_\lambda \cdot]_{\phi_\L}, \phi_\L)$ with respect to the representation $(\rho_{\phi_{\M}}, \M, \phi_{\M})$.
	
 Let's now introduce the notion of graph of a $\mathbb C[\partial]$-module homomorphism $\T : \M\to \L$. We denote the graph of $\T$ by $Gr(\T) = \{(\T(m),m)~|~m \in \M\}.$ The set $ Gr(\T) $ helps us to characterize the $\mathcal{O}$-operator $\T$ in terms of a Hom-Lie
 conformal subalgebra structure, consider the following proposition.
	\begin{prop}Let $(\L, [\cdot _\l \cdot], \a)$ be a Hom-Lie conformal algebra and $(\rho, \M, \b)$ is its representation. A map $\T : \M \to \L$ is an $\mathcal{O}$-operator if and only if the graph of the map $\T$, $Gr(\T)$ as defined above, is the Hom-Lie conformal subalgebra of a semi-direct product Hom-Lie conformal algebra $(\L\oplus \M, [\cdot_\l\cdot]_\rho, \a+\b)$ defined in the Proposition \eqref{directsum}.
	\end{prop}	
\begin{proof} In order to show that $Gr(\T)$ is subalgebra of $(\L\oplus \M, [\cdot_\l\cdot]_\rho, \a+\b)$, we first consider that $\T$ is an $\mathcal{O}$-operator. For any $m, n\in \M$, the $Gr(\T)$ is closed under the $\l$-bracket given by
	\begin{align*}{[(T(m), m)_{\l}(T(n), n)]}_{\L\oplus \M}=& ([{\T(m)}_\l\T(n)], \rho(\T(m))_\l n- \rho (\T(n))_{-\p-\l} m)\\=& (\T(\rho(\T(m))_\l n- \rho(\T(n))_{-\p-\l} m), \rho(\T(m))_\l n- \rho (\T(n))_{-\p-\l} m)\\ \in& Gr(\T).\end{align*}
	Further more, we have\begin{align*}
	(\a+\b)(\T(m), m)= (\a(\T(m)), \b(m))= (\T(\b(m)), \b(m))\in Gr(\T).
	\end{align*}Thus, $Gr(\T)\subset (\L\oplus \M, [\cdot_\l\cdot]_\rho, \a+\b)$.\\
	Secondly, we show that if $Gr(\T)$ is a subalgebra of $\L\oplus \M$, then $\T$ should be an $\mathcal{O}$-operator. As \begin{align*} {[(T(m), m)_{\l}(T(n), n)]}_{\L\oplus \M}=& ([{\T(m)}_\l \T(n)], \rho(\T(m))_\l n- \rho(\T(n))_{-\p-\l} m) \in Gr(\T).
	\end{align*}
	It implies that $$[{\T(m)}_\l\T(n)]= \T(\rho(\T(m))_\l n-\rho (\T (n))_{-\p-\l} m)$$ also $(\a+\b)(\T(m), m)=(\a(\T(m)), \b(m))\in Gr(\T)$. This condition is equivalent to the fact that $\a\T= \T\b$. Thus, $\T$ is an $\mathcal{O}$-operator on the Hom-Lie conformal algebra.
	\end{proof}
	It is well-known that $\mathcal{O}$-operators on the Lie conformal algebras can be characterized in terms of the Nijenhuis operators on the Lie conformal algebras. In the next result, we show that  $\mathcal{O}$-operators on Hom-Lie conformal algebras can also be characterize in terms of the Nijenhuis operators. Let us first recall from \cite{Zhao-Chen-Yuan, Asif-Wang-Wu} that a $\mathbb{C}[\p]$-linear map $\N : \L\to \L$ is called a Nijenhuis operator on the Hom-Lie conformal algebra $(\L,[ \cdot_\l\cdot ], \alpha)$ if
	\begin{align*}
	\a\N&=\N\a,\\
	[\N(x)_{\lambda}\N(y)]&=\N([\N(x)_{\lambda}y]- [\N(y)_{-\partial-\lambda} x] - \N([x_{\lambda} y])),~\forall~x,y\in \L,~\l\in \mathbb{C}.
	\end{align*}
	Then, we have the following characterization of $\mathcal{O}$-operators on Hom-Lie conformal algebras in terms of Nijenhuis operators.
	
	\begin{prop}A $\mathbb{C}[\p]$-module homomorphism $\T:\M\to \L$ is an $\mathcal{O}$-operator on $(\L, [\cdot_\lambda\cdot],\alpha)$ with respect to the representation $(\rho,\M,\b)$ if and only if the operator 
		\begin{align*}
		\N_{\T}=\begin{bmatrix}
		0 & \T \\
		0 & 0
		\end{bmatrix}: \L\oplus \M\to \L\oplus \M
		\end{align*}
		 is a Nijenhuis operator on the semi-direct product Hom-Lie conformal algebra $(\L\oplus \M, [ \cdot_\lambda\cdot]_\rho, \alpha\oplus \beta)$.
	\end{prop}
	\begin{proof}Consider that for any $x, y \in \L$, $m, n \in \M$ and $\l,\m\in \mathbb{C}$, we have the following \begin{equation*}[\N_{\T} (x + m)_{\lambda} \N_{\T}(y + n)]_{\rho}= [(0 +\T(m))_{\lambda}( 0+ \T(n)) ]_{\rho}= [\T(m)_\lambda \T(n)], \end{equation*} and\begin{equation*}
		\begin{aligned}
		&\N_{\T}([\N_{\T} (x + m)_\lambda (y + n)]_{\rho} - [\N_{\T}(y + n)_{-\partial-\lambda} (x + m)]_{\rho} - \N_{\T} ([(x + m)_{\lambda} (y + n)]_{\rho}))\\=& \N_{\T}
		(([\T(m)_{\lambda} y] + \rho (\T(m))_{\lambda}(n))- ([\T(n)_{-\partial-\lambda} x] + \rho(\T(n))_{-\partial-\lambda} (m))\\&-(0 + \T(\rho(x)_{\lambda}(n)- \rho(y)_{-\partial-\lambda}(m))))\\=&\N_{\T}
		([\T(m)_{\lambda} y] -[\T(n)_{-\partial-\lambda} x]-\T(\rho(x)_{\lambda}(n)- \rho(y)_{-\partial-\lambda}(m)+\rho (\T(m))_{\lambda}(n))-\rho(\T(n))_{-\partial-\lambda} (m))\\=&\T(\rho(T(m))_{\lambda}(n)- \rho(\T(n))_{-\partial-\lambda}(m)).
		\end{aligned}\end{equation*} By the above equations, it is clear that the $\N_{\T}$ is equivalent to an $\mathcal{O}$-operator $\T$.
	\end{proof}
		\begin{defn}\label{def3.7}Let $\T : \M \to \L$ and $\T' : \M \to \L$ are two $\mathcal{O}$-operators on the Hom-Lie conformal algebra $(\L,[\cdot_\l\cdot],\alpha)$ with respect to the representation $(\rho,\M,\beta)$. A homomorphism from $\T$ to $\T'$ is given by a pair $(\phi_{\L}, \phi_{\M})$, consisting of a Hom-Lie conformal algebra homomorphism $\phi_{\L} : \L \to \L$ and a linear map $\phi_{\M}: \M\to \M$ such that following conditions are satisfied
		\begin{eqnarray}\begin{aligned}\T' \phi_{\M}=& \phi_{\L}\T,\\
		 \phi_{\M}\beta=& \beta\phi_{\M},\\
	\phi_{\M}(\rho(x)_{\lambda}(m))=& \rho(\phi_{\L}(x))_{\lambda}(\phi_{\M} (m)),~~\forall~~x\in \L,~~m\in \M,~~\l\in \mathbb{C}.
		\end{aligned}
		\end{eqnarray}
	\end{defn}
	\subsection{ Cohomology of $\mathcal{O}$-operator on the Hom-Lie conformal algebra:}
We first need to construct a graded Lie algebra for a Hom-Lie conformal algebra with a representation whose Maurer-Cartan elements characterize the $\mathcal{O}$-operators. It then follows that a given $\mathcal{O}$-operator gives rise to a differential on this graded Lie algebra and there is a one-to-one correspondence between the set of Maurer-Cartan elements in the resulting differential graded Lie algebra and the set of deformations of this $\mathcal{O}$-operator. Let's recall some useful definitions and notions.
Let $(\L,[\cdot_\l\cdot],\a)$ be a Hom-Lie conformal algebra with a representation $(\rho, \M, \b)$. Then, we have a graded Lie algebra $$(\L^*= C^{*}_{\a,\b}(\L \oplus \M, \L \oplus \M), {[\cdot,\cdot]}^{\a+\b}_{NR})$$ associated to the pair $(\L\oplus \M,\a+\b)$. Then, by Proposition \eqref{imp}, an element $m_c+ \rho\in C^{2}_{\a, \b}(\L\oplus \M, \L\oplus \M)$ satisfying $[m_c+ \rho, m_c+ \rho]^{\a+ \b}_{NR}=0$ is a Maurer-Cartan element of the graded Lie algebra $\L^*$. We can define a differential map or coboundary operator $d_{m_c+\rho}: \L^*\to \L^{*+1}$ by $d_{m_c+\rho}:= [m_c+\rho,-]^{\a+\b}_{NR}$ corresponding to our graded Lie algebra. As a result we obtain a differential graded Lie algebra $$(\L^*= C^{*}_{\a, \b}(\L \oplus \M, \L \oplus \M), {[\cdot, \cdot]}^{\a+ \b}_{NR}, d_{m_c+ \rho}).$$
 There are three steps to construct the differential graded Lie algebra whose Maurer-Cartan element characterizes the $\mathcal{O}$-operator:\\
	\textbf{Step 1:} In this step, we define the graded space of $p$-cochains. Let $(\rho, \M, \b)$ be a representation of a Lie conformal algebra $(\L, [\cdot_\l\cdot],\a)$. Now to develop the cohomology theory of $\mathcal{O}$-operator on Hom-Lie conformal algebra with respect to given representation, we consider the following space of $p$-cochains $$C^{*}_{\b, \a}(\M, \L)= \bigoplus_{p\geq1} C^{p}_{\b, \a}(\M, \L).$$ Where, the space $C^{p}_{\b, \a}(\M, \L)$ consists of all the $\mathbb{C}$-linear maps: $$f_{\l_1, \l_2, \cdots, \l_{p-1}}: \M^{\otimes p}\to\L[\l_1, \l_2, \cdots, \l_{p-1}],$$ that along with sesqui-linearity condition (see \cite{Yuan3}), satisfy the following conditions:
	\begin{align*}
	\a(f_{\l_1, \cdots, \l_{p-1}}(m_1, m_2, \cdots, m_p))= f_{\l_1, \cdots, \l_{p-1}}(\b(m_1), \b(m_2), \cdots, \b(m_p)), ~ \forall ~m_1, m_2, \cdots , m_p\in \M.
	\end{align*}
\textbf{Step 2:} In this  step, we define a skew-symmetric graded Lie bracket by
	$$\{\!\!\{\cdot,\cdot\}\!\!\}: C^{q}_{\b, \a} (\M, \L)\times  C^{p}_{\b, \a} (\M, \L) \to C^{p+ q}_{\b, \a} (\L, \L)$$ as follows:
\begin{align*}
		\{\!\!\{f,g\}\!\!\}= (-1)^p[[m_c+\rho, f]^{\a+\b}_{NR}, g]^{\a+\b}_{NR}. \end{align*}
		By definition of the graded Lie bracket $[\cdot,\cdot]^{\a+\b}_{NR}$, the bracket $\{\!\!\{\cdot,\cdot\}\!\!\}$ can be written as
		follows: 
	\begin{align*}&\{\!\!\{f,g\}\!\!\}_{\l_1,\cdots\l_{p+q-1}}(m_1, m_2,\cdots, m_{p+q})
	= \sum_{\tau\in S_{p,1, q-1}}(-1)^{\tau} f_{\l_{\tau(1)}+\cdots+\l_{\tau(p-1)}+\l_{\tau(p)},\l_{\tau(p+1)},\cdots,\l_{p+q-1}}
	\\&(\rho(g_{\l_{\tau(1)},\cdots,\l_{\tau(p-1)}}(m_{\tau(1)},\cdots,m_{\tau(p)}))_{\l_{\tau(1)}+\cdots+\l_{\tau(p-1)}+\l_{\tau(p)}}\b^{p-1}(m_{\tau(p+1)}), \b^{p}(m_{\tau(p+2)}), \cdots,\b^{p}(m_{\tau(p+q)}))
	\\&+(-1)^{pq}(\sum_{\tau\in S_{q,p}}(-1)^\tau\\& [\a^{p-1}f_{\l_{\tau(1)},\cdots\cdots,\l_{\tau(q-1)}}(m_{\tau(1)},\cdots, m_{\tau(q)})_{{\tau(1)}+\cdots+\l_{\tau(p-1)}+\l_{\tau(p)}}
	\a^{p-1}g_{\l_{\tau(p+1)},\cdots,\l_{\tau(p+q-1)}}(m_{\tau(p+1)},\cdots, m_{\tau(p+q)})]
	\\&- \sum_{\tau\in S_{q,1,p-1}}(-1)^{\tau} g_{\l_{\tau(1)}+\cdots+\l_{\tau(q-1)}+\l_{\tau(q)},\l_{\tau(q+1)},\cdots,\l_{q+p-1}}
	\\&(\rho(f_{\l_{\tau(1)},\cdots,\l_{\tau(q-1)}}(m_{\tau(1)},\cdots , m_{\tau(q)}))_{\l_{\tau(1)}+\cdots+\l_{\tau(q)}}\b^{q-1}(m_{\tau(q+1)}), \b^{q}(m_{\tau(q+2)}),\cdots, \b^{q}(m_{\tau(p+q)}))).
	\end{align*}
Moreover, for any $\T\in C^1_{\b,\a}(\M,\L)$ satisfying $\a\T=\T\b$, above bracket becomes:
	\begin{align*}\{\!\!\{\T,\T\}\!\!\}_{\l_1}(m_1, m_2)= 2(\T (\rho(\T m_{1})_{\l_1}m_{2})-\T(\rho(\T (m_{2}))_{-\p-\l_1}m_{1})-[{\T (m_{1})}_{\l_1} \T (m_{2})]).\end{align*}In turn, it follows that $\{\!\!\{\T, \T\}\!\!\}=0$ if and only if $\T : \M \to \L $ is an $\mathcal{O}$-operator on Hom-Lie conformal algebra $(\L, [\cdot_\l\cdot],\a)$ with respect to the representation $(\rho,\M,\b)$. Thus we have the following theorem.
	\begin{thm}\label{maurer}The graded space $C^*_{\b,\a}(\M, \L)$ of $p$-cochains forms a graded Lie algebra with the graded Lie bracket $\{\!\!\{-,-\}\!\!\}$. A linear map $\T: \M \to \L$ satisfying $\a\T=\T\b$ is an $\mathcal {O}$-operator on Hom-Lie conformal algebra $(\L, [\cdot_\l\cdot],\a)$ with respect to the representation $(\M,\b,\rho)$ if and only if $\T\in C^1_{\b,\a}(\M,\L)$ is a Maurer-Cartan element of the graded Lie algebra $(C^*_{\b,\a}(\M, \L), \{\!\!\{-,-\}\!\!\})$, i.e, $\{\!\!\{\T,\T\}\!\!\}=0$.
	\end{thm}

\textbf{Step 3:}In this step, we construct the differential or coboundary map for the graded Lie algebra.\begin{rem}From the Theorem \eqref{maurer}, $\T\in C^1_{\b,\a}(\M, \L)$ is a Maurer-Cartan element of the graded Lie algebra $(C^*_{\b,\a}(\M, \L), \{\!\!\{-,-\}\!\!\})$. Then, the $\mathcal{O}$-operator $\T$ induces a differential $d_\T:= \{\!\!\{T,-\}\!\!\}$ on the graded Lie algebra $(C^*_{\b,\a}(\M, \L), \{\!\!\{-,-\}\!\!\})$, which makes it a differential graded Lie algebra.
\end{rem}The cohomology of cochain complex $(C^*_{\b,\a}(\M,\L), d_\T)$ with respect to an $\mathcal{O}$-operator $\T$ is called the cohomology of the $\mathcal{O}$-operator $\T$ on the Hom-Lie conformal algebra,denoted by $\mathcal H_{\b,\a}^*(\M,\L)$.
\subsection{ Cohomology of $\mathcal{O}$-operators in terms of Hom-Lie conformal algebra's cohomology:} Now, we describe the cohomology of an $\mathcal{O}$-operator $\T$ on Hom-Lie conformal algebras $\mathcal H_{\b,\a}^*(\M,\L)$ in terms of Hom-Lie conformal algebra's cohomology with coefficients in a representation $\mathcal H_{\a,\b}^*(\L,\M)$.

We next recall the notion of a Hom-pre-Lie conformal algebra and the graded Lie algebra whose Maurer-Cartan elements characterize pre-Lie conformal algebra structures. We show that there is a close relationship between these two graded Lie algebras.
\begin{defn}
	A Hom-pre-Lie conformal algebra is a triple $(\M,*_\l,\b)$, where $\M$ is a $ \mathbb{C}[\p] $-module equipped with a bilinear map $*_{\lambda} : \M \otimes\M \to \M[\l]$ and a $\mathbb{C}$-linear map $\b :\M \to \M$ such that following equations are satisfied, for $m,n,r\in \M$:
	\begin{align*}
	\beta(m *_{\lambda} n)&=\beta(m)*_{\lambda}\beta(n),\\
	(m*_\lambda  n)*_{\lambda+ \mu}\beta(r)-\beta(m)*_\lambda(n*_{\mu}r)&= (n*_{\m}m)*_{\lambda+\mu}\beta(r)-\beta(n)*_{\mu}(m*_{\lambda}r). \end{align*}
\end{defn}
An $\mathcal{O}$-operator on a Hom-Lie conformal algebra induces a Hom-pre-Lie  conformal algebra. In particular, we have
the following straightforward proposition.
\begin{prop}
	Let $\T:\M \to \L$ be an $\mathcal{O}$-operator on the Hom-Lie conformal algebra $(\L, [\cdot_\l\cdot], \alpha)$ with
	respect to the representation $( \rho,\M,\beta)$. Then, the $\mathcal{O}$-operator induces a Hom-pre-Lie conformal algebra $(\M,*_{\l}^ \T,\beta)$, where $*_{\l}^\T$ is given by \begin{align}\label{prelie}
m*_{\l}^{\T} n = \rho(\T m)_{\lambda}(n), \quad for~ all ~ m,n \in \M.\end{align}\end{prop}
\begin{proof}Consider that, for $m,n,r\in \M$, we have\begin{align*}
	&(m*_\lambda  n)*_{\lambda+ \mu}\beta(r)- \beta(m)*_\lambda(n*_{\mu}r)- (n*_{\m}m)*_{\lambda+ \mu}\beta(r)+ \beta(n)*_{\m}(m*_{\lambda}r)\\
	=&\rho(\T(\rho(\T m)_\l n))_{\l+\m} \b(r)
	-\rho(\T(\b(m)))_{\l}\rho(\T (n))_\m r
	-\rho(\T(\rho(\T(n))_{\m}m))_{\l+\m}\b(r)\\&
	+\rho(\T\b(n))_\m(\rho(\T m)_\l r)\\=&
	\rho([\T(n)_{\l}\T(m)])_{\l+\m}\b(r) +\rho(\T(\rho(\T(m))_{\l}n -\rho(\T(n))_{\m}m))_{\l+\m}\b(r).
\end{align*}
	The Eq. \eqref{prelie} defines a Hom-pre-Lie conformal algebra if and only if
	\begin{align*}
	[\T(m)_{\l}\T(n)] +\T(\rho(\T(m))_{\l}n -\rho(\T(n))_{\m}m)\in Ker(\rho),
	\end{align*}
	where $Ker(\rho)=\{x \in \L|~~ \rho(x) = 0\}$. In particular, for any $\mathcal{O}$-operator $\T : \M \to \L$ associated to $( \rho,\M,\b)$, Eq. \eqref{prelie} defines a Hom-pre-Lie conformal algebra on $\M$.\end{proof}
If $(\M,*_{\l },\beta)$ is a Hom-pre-Lie conformal algebra, then the commutator bracket $$[m_\l n]^c= m*_{\lambda}n- n*_{-\partial-\lambda} m, \quad\forall~m, n\in\M$$ gives a Hom-Lie conformal algebra structure $\M ^{c}_{\beta}:=(\M, [\cdot_\l\cdot]^c,\beta)$. It is called the sub-adjacent Hom-Lie conformal algebra of the Hom-pre-Lie conformal algebra $(\M, *_{\lambda}, \beta)$.
\begin{prop}Let $\T:\M\to \L$ is an $\mathcal{O}$-operator on the Hom-Lie conformal algebra $(\L,[\cdot_\l\cdot],\alpha)$ with respect to the representation $(\rho,\M,\b)$. Let us define a $\mathbb{C}$-linear map $\rho_{\T}: \M\to gc(\L)$ given by \begin{align*}
	\rho_{\T}(m)_\l(x):=[{\T (m)}_\l x]+ \T(\rho(x)_{-\p -\l } (m)),\quad\textit{ for all } m\in \M ~and~ x \in \L.
	\end{align*} Then, the triple $(\rho_{\T},\L,\a)$ is a new representation of the sub-adjacent Hom-Lie conformal algebra $\M^{c}_{\beta}$.\end{prop}
\begin{proof}First of all we  show that $\rho_\T (\beta(m))_{\lambda}(\alpha(x))= \alpha(\rho_\T (m)_{\lambda}(x))$. This identity holds by using the fact of Definition \eqref{def3.7}. In fact  \begin{align*}
	\rho_\T (\beta(m))_{\lambda}(\alpha(x))&=[\T(\beta(m))_\lambda (\alpha(x))]+ \T(\rho(\alpha(x))_{-\partial-\lambda} (\beta(m)))\\&=[\alpha \T(m)_\lambda (\alpha(x))]+ \T(\beta(\rho(x)_{-\partial-\lambda} (m)))\\&=\alpha[\T(m)_\lambda x]+ \alpha \T(\rho(x)_{-\partial-\lambda} (m))\\&=\alpha([\T(m)_\lambda x]+  \T(\rho(x)_{-\partial-\lambda} (m)))\\&=\alpha(\rho_\T (m)_{\lambda}(x)).\end{align*}
	Next, we use the properties of an $\mathcal{O}$-operator to obtain the following expressions:
	\begin{align*}
	&\rho_{\T}([m_{\lambda}n]^{c})_{\lambda+\mu}(\alpha(x))\\=& [\T([m_{\lambda}n]^{c})_{\lambda+\mu} \alpha(x)]+ \T(\rho(\alpha(x))_{-\p-\m} ([m_{\lambda}n]^{c}))\\
	=&[\T(m*_{\lambda}n-n*_{-\p-\l}m)_{\l+\m} \alpha(x)]+ \T(\rho(\a(x))_{-\p-\m} (m*_{\lambda}n-n*_{-\p-\l}m))\\=&[\T(\rho(\T m)_{\lambda}(n)-\rho (\T n)_{-\partial-\lambda}(m))_{\lambda+\mu} \alpha(x)]\\&+ \T(\rho(\alpha(x))_{-\p-\mu} (\rho(Tm)_{\lambda}(n)-\rho (\T n)_{-\partial-\lambda}(m)))\\=&[[\T m_{\lambda}\T n]_{\lambda+\mu} \alpha(x)]+\T(\rho(\alpha(x))_{-\p-\mu} (\rho(\T m)_{\lambda}(n)-\rho (\T n)_{-\partial-\lambda}(m)))\\=&[[\T m_{\lambda}\T n]_{\lambda+\mu} \alpha(x)]+ \T(\rho(\alpha(x))_{-\p-\mu} (\rho(\T m)_{\lambda}(n)))-\T(\rho(\alpha(x))_{-\p-\m}\rho (\T n)_{-\partial-\lambda}(m)).
	\end{align*}
	Now
	\begin{equation}\label{eq18}
	\begin{aligned}
	&\rho_\T(\beta(m))_{\lambda}\rho_{\T}(n)_{\mu}(x)\\=&\rho_{\T}(\beta(m))_\lambda([\T n_\mu x]+ \T(\rho(x)_{-\partial-\mu}(n)))\\
	=&[\T(\beta(m))_\lambda([\T n_\mu x]+ \T(\rho(x)_{-\partial-\mu}(n)))]+\T(\rho([\T n_\mu x]+ \T(\rho(x)_{-\partial-\mu}(n)))_{-\p-\lambda} \beta(m))\\
	=&[\T\beta(m)_\lambda[\T n_\mu x]]+ [\T\beta(m)_\lambda \T(\rho(x)_{-\partial-\mu}(n))]+\T(\rho([\T n_\mu x])_{-\p-\l} \beta(m))\\&+\T\rho(\T(\rho(x)_{-\partial-\mu}(n))_{-\p-\l} \beta(m))\\
	=&[\T\beta(m)_\lambda[\T n_\mu x]]+ \T\rho (\T\beta(m))_{\lambda}(\rho(x)_{-\partial-\mu}(n))-\T\rho(\T(\rho(x)_{-\partial-\mu}(n))_{-\p-\l}(\b(m)))\\&+\T(\rho([\T n_\mu x])_{-\p-\l} \beta(m))+ \T\rho(\T(\rho(x)_{-\partial-\mu}(n))_{-\p-\l} \beta(m))\\
	=&[\alpha (\T(m))_\lambda[\T n_\mu x]]+  \T\rho (\a\T(m))_{\lambda}(\rho(x)_{-\partial-\mu}(n))+\T(\rho([\T n_\mu x])_{-\p-\l} \beta(m)).
	\end{aligned}
	\end{equation}
	Similarly, we have
	\begin{equation}\label{eq19}
	\begin{aligned}
	&\rho_\T(\beta(n))_{\m}\rho_{\T}(m)_{\l}(x)\\=&[\alpha (\T(n))_\m[\T m_\l x]]+  \T\rho (\a\T(n))_{\m}(\rho(x)_{-\p-\l}(n))+\T(\rho([\T m_\l x])_{-\p-\m} \beta(n)).
	\end{aligned}
	\end{equation}
	Now subtraction of Eq. \eqref{eq19} from Eq. \eqref{eq18}, and by the use of definition of Hom-Lie conformal algebra and Hom-Lie conformal representation, we get  
	\begin{equation}\rho_\T(\beta(m))_{\lambda}\rho_{\T}(n)_{\mu}(x)-\rho_\T(\beta(n))_{\m}\rho_{\T}(m)_{\l}(x)=\rho_{\T}([m_{\lambda}n]^{c})_{\lambda+\mu}(\alpha(x))\end{equation}
	for all $m,n \in \M$ and $x \in \L$. Thus, the triple $({\rho}_{\T}, \L,\a)$ is a representation of sub-adjacent Hom-Lie conformal algebra $\M^{c}_{\beta}.$
\end{proof}
With the above notations, let us now consider the cochain complex $(C^*_{\b,\a}(\M, \L), \delta_{\b,\a})$, that we also call as modified cochain complex of the subadjacent Hom-Lie conformal algebra $\M ^c_\b$ with coefficients in the representation $(\rho_\T,\L, \a )$. Recall that the
differential $\delta_{\b,\a}$ is given by
\begin{equation}\label{differential}
\begin{aligned}
&\delta_{\b,\a} f_{\l_1,\cdots,\l_{p}}(m_1, m_2,\cdots, m_{p+1})\\=& \sum_{i= 1}^{p+ 1} (-1)^{i+1}\rho_\T((\b^{p-1}(m_i))_{\l_i}f_{\l_1,\cdots,\hat{\l_i},\cdots\l_p}(m_1, m_2,\cdots,\hat{m_i},\cdots, m_{p+ 1}))\\&+
\sum_{i< j} (-1)^{i+ j}f_{\l_i+\l_j,\l_1,\cdots,\hat{\l_i},\cdots,\hat{\l_j},\cdots,\l_p}([{m_i}_{\l_i} m_j]^c, \b(m_1), \cdots, \hat{m_i}, \cdots, \hat{m_j}, \cdots,\b (m_{p+1})).
\end{aligned}
\end{equation}
Thus, there are two different differentials $\delta_{\b,\a}$ and $\delta_\T := \{\!\!\{\T,-\}\!\!\}$ on the graded vector spaces
$C^*_{\b,\a}(\M, \L)$. The following proposition shows that both of these differentials yield the same
cohomology.
\begin{prop}
	Let $\T : \M \to \L$ is an $\mathcal{O}$-operator on the Hom-Lie conformal algebra $(\L, [\cdot_\l\cdot],\a)$ with
	respect to the representation $(\M, \b, \rho)$. Then, the differentials $\delta_{\b, \a}$ and $\delta_{\T}$ on $C^*_{\b, \a}(\M, \L)$ are related by
	$$\delta_\T(P) = (-1)^p\{\!\!\{\T,P\}\!\!\}, \quad for\quad P\in C^p_{\b,\a}(\M, \L) \quad and \quad p \geq 1.$$
\end{prop}
\begin{proof}For $P\in C^p_{\a,\b}(\M, \L)$ and $p\geq1$, we have
	\begin{equation}\label{eqdifop}
	\begin{aligned}&\{\!\!\{\T,P\}\!\!\}_{\l_1,\cdots,\l_{p}}(m_{1},\cdots,m_{p+1})\\
	=&\sum_{\tau\in S_{p,1,0}} (-1)^{\tau} \T (\rho(P_{\l_{\tau(1)},\cdots,\l_{\tau(p-1)}}(m_{\tau(1)},\cdots,m_{\tau(p)}))_{\l_{\tau(1)}+\cdots+\l_{\tau(p)}}\b^{p-1}(m_{\tau(p+1)}))
	\\& +(-1)^{p}(\sum_{\tau\in S_{1,p}}(-1)^\tau [\a^{p-1} \T(m_{\tau(1)})_{\l_{\tau(1)}}  P_{\l_{\tau(2)},\cdots,\l_{\tau(p)}} (m_{\tau(2)},\cdots, m_{\tau(p+1)})]
	\\&-\sum_{\tau\in S_{1,1, p-1}}(-1)^{\tau}P_{\l_{\tau(1)}+\l_{\tau(2)},\l_{\tau(3)},\cdots,\l_{\tau(p)}}(\rho(\T(m_{\tau(1)})_{\l_{\tau(1)}}m_{{\tau}(2)},\b(m_{{\tau}(3)}),\cdots,\b(m_{{\tau}(p+ 1)})))\\
	=&(-1)^{p}(\sum_{i=1}^{p+1} (-1)^{i+1} \T  (\rho(P_{\l_1,\cdots,\hat{\l_i},\cdots,\l_{p}}(m_{1},\cdots,\hat{m_{i}},\cdots,m_{p+1}))_{\l_1+\cdots+\hat{\l_i}+\cdots+\l_{p+1}}\b^{p-1}(m_{i}))\\&+\sum_{i=1}^{p+1}(-1)^{i+1} [\a^{p-1} \T(m_{i})_{\l_i}P_{\l_1, \cdots,\hat{\l_i},\cdots, \l_{p}} (m_{1},\cdots,\hat{m_{i}},\cdots, m_{ p+1})]\\&
	+\sum_{i<j}(-1)^{i+j}P _{\l_i+\l_j,\l_1,\cdots,\hat{\l_i},\cdots,\hat{\l_j},\cdots,\l_{p}}(\rho(\T(m_{i}))_{\l_i}m_j-\rho(\T(m_{j}))_{-\p-\l_i}m_i,\b(m_1),\cdots,  
	 \b(\hat{m_i}), \\&\cdots,\b(\hat{m_j}),\cdots, \b(m_{p+1})))
	\\=& (-1)^{p}(\sum_{i= 1}^{p+ 1} (-1)^{i+1}\rho_\T(\b^{p-1}(m_i))_{\l_i}P_{\l_1, \cdots,\hat{\l_i},\cdots, \l_{p}}(m_1, m_2,\cdots,\hat{m_i},\cdots,m_{p+1})\\&+\sum_{i<j}(-1)^{i+ j}P_{\l_i+\l_j,\l_1,\cdots,\hat{\l_i},\cdots\hat{\l_j},\cdots,\l_{p}}([{m_i}_{\l_i} m_j]^c, \b(v_1), \cdots, \b(\hat{m_i}), \cdots, \b(\hat{m_j}), \cdots, \b(m_{p+1})))\\=&(-1)^{p}\delta_{\a,\b} P_{\l_1,\cdots\l_p}(m_1, m_2,\cdots, m_{p+1}).\end{aligned}\end{equation}
	Hence, $\delta_{\a, \b}(P) = (-1)^p\{\!\!\{T, P\}\!\!\}= (-1)^p\delta_{\T}(P)$, for any $P\in C^p_{\a,\b}(\M, \L)$ and $p \geq 1.$\end{proof}
\section{ Deformation of an $\mathcal{O}$-operator on regular Hom-Lie conformal algebras.}
\subsection{ Linear deformations:}Let $\mathcal{T}:\mathcal{M}\rightarrow \mathcal{L}$ be an $\mathcal{O}$-operator on Hom-Lie conformal algebra $(\mathcal{L},[\cdot_\l\cdot],\alpha)$ with a representation $(\rho, \M,\beta)$.
Let us consider $\mathcal{T}_{t}=\mathcal{T}+t\mathfrak{T}$, where $\mathfrak{T}\in C^{1}_{\b,\a}(\mathcal{M},\mathcal{L})$. If $\mathcal{T}_{t}$ is an $\mathcal{O}$-operator on the Hom-Lie conformal algebra $(\L,[\cdot_\l\cdot],\alpha)$ with a representation $(\rho, \M,\beta)$, we call $\mathcal{T}_{t}$ a linear deformation of $\mathcal{T}$ generated by the element $\mathfrak{T}$. The map $\mathcal{T}_{t}= \mathcal{T}+t\mathfrak{T}$ is a linear deformation of $\T$ if it satisfies:
\begin{align*}
\mathcal{T}_{t}\circ \beta &= \alpha\circ \mathcal{T}_{t}\\
[\mathcal{T}_{t}(m)_{\lambda} \mathcal{T}_{t}(n)] &= \mathcal{T}_{t}(\rho(\mathcal{T}_{t}(m)-\rho(\mathcal{T}_{t}(n)_{-\partial-\lambda}(m),\quad  \forall ~m, n \in \mathcal{M}.
\end{align*}This equation is equivalent to the following expressions:
\begin{align}
\label{b1}\mathfrak{T}\circ \beta &= \alpha\circ \mathfrak{T},\\\label{b2}
[\mathfrak{T}(m)_{\lambda} \mathfrak{T}(n)] &= \mathcal{T}_{t}(\rho(\mathfrak{T}(m))_\l n-\rho(\mathfrak{T}(n))_{-\partial-\lambda}m),\\\label{b3}
[\mathcal{T}(m) _{\lambda}\mathfrak{T}(n)] + [\mathfrak{T}(m) _{\lambda}\mathcal{T}(n)] &=\mathfrak{T}(\rho(\mathcal{T}(m)_\lambda n-\rho(\mathcal{T}(n)_{-\partial-\lambda} m)+\T(\rho(\mathfrak{T}(m))_{\lambda} n\\&\nonumber-\rho(\mathfrak{T}(n))_{-\partial-{\lambda}} m),
\end{align}for all $m, n \in \mathcal{M}$. It is clear that $\T_{t}$ is linear deformation of $\T$ if and only if Eqs. \eqref{b1}, \eqref{b2} and \eqref{b3} hold. The Eqs. \eqref{b1} and \eqref{b2} shows that $\mathfrak{T}$ is an $\mathcal{O}$-operator on the Hom-Lie conformal algebra $(\mathcal{L}, [\cdot_{\lambda}\cdot],\alpha)$ with a representation $(\rho,\mathcal{M},\beta)$.
Though from the Eq. (\ref{b3}) is nothing but $d_{\mathcal{T}} (\mathfrak{T}) = 0$. 
\begin{defn}
	Two linear deformations $\mathcal{T}_{t}^{1}= \mathcal{T} + t\mathfrak{T}_{1}$ and $\mathcal{T}_{t}^{2} = \mathcal{T} + t\mathfrak{T}_{2}$ are said
	to be equivalent if there exist an element $x \in \mathcal{L}$ such that $\alpha(x) =  x$ and the pair  $(Id_{\mathcal{L}} + tad^\dagger_{x}, Id_{\mathcal{M}} + t\rho(x)^\dagger)$ is a homomorphism of $\mathcal{O}$-operators from $\mathcal{T}_{t}^{1}$ to $\mathcal{T}_{t}^{2}$.
\end{defn}
Let us recall from Definition \eqref{def3.7} that the pair $(Id_\L + tad^{\dagger}_x, Id_\M+ t\rho(x)^{\dagger})$ is a homomorphism of $\mathcal{O}$-operators from $\T^1_t \to \T^2_t$, if the following conditions are satisfied:
\begin{enumerate}
	\item \label{con1} The map $Id_\L+ tad^\dagger(x)$ is a Hom-Lie conformal algebra homomorphism,
	\item \label{con2}$\b(Id_{\L}+t\rho(x)^{\dagger})= (Id_{\L} + t\rho(x)^{\dagger})\b,$
	\item \label{con3} $(\T + t\mathfrak{T}_2)(Id_\M + t\rho(x)^{\dagger})=(Id_{\L}+ tad^{\dagger}_x)(\T + t\mathfrak{T}_1),$
	\item \label{con4} $\rho ((Id_{\L}+ tad^{\dagger}_x)(y))_\m((Id_\L +t\rho(x)^\dagger)(m))=(Id_\L+ t\rho(x)^\dagger) (\rho(y)_\m(m)), \forall  y\in \L ~{ and } ~m \in \M.$
\end{enumerate}
\textbf{Condition \eqref{con1}$ \implies $} Let $(Id_\L + tad^\dagger_x , Id_\M+ t\rho(x)^{\dagger})$ be a homomorphism from $\T^1_t \to \T^2_t$ then $Id_\L + tad^\dagger_x $ is a morphism of $\L$. Thus we have
$$(Id_\L + tad^\dagger_x)[y_\m z] = [(Id_\L + tad^\dagger_x)(y)_\mu (Id_\L+ tad^\dagger_x)(z)],$$  for all $y, z \in \L$.
On comparing the coefficients of $t^2$ on both sides of the above identity, we get
$$[[x_\l\a^{-1}(y)]_\m [x_\l\a^{-1}(z)]] = 0,  \quad \forall ~y, z \in \L.$$
Clearly, invertibility of $\a$ implies that the element $x$ satisfies \begin{equation}\label{eq2con}
[[x_\l y]_\m[x_\l z]] = 0, \quad\forall~ y, z \in \L.
\end{equation}
\textbf{Condition \eqref{con2}$ \implies $} Holds, since $\a(x) = x$.\\
\textbf{Condition \eqref{con3}$ \implies $}
By using the Definition \eqref{def3.7}, we get \begin{align*}(\T + t\mathfrak{T}_1)(Id_\M + t\rho(x)^\dagger)(m) = (Id_\L + tad_x )(\T + t\mathfrak{T}_2)(m),~~ \forall ~~m\in \M,
\end{align*} which implies that
\begin{align}\label{eqsaniaasif}
(\mathfrak{T}_2- \mathfrak{T}_1)(m) =& \T (\rho(x)_\l(\b^{-1}(m)) )+ [\T(\b^{-1}m)_{-\p-\l}x] =\delta_{\T}(x)(m),\\
\label{eqsania}\mathfrak{T}_{1}\rho(x)_\l(\b^{-1}m) =& [x_\l \mathfrak{T}_2 (\b^{-1}m)], ~\forall ~m \in \M.
\end{align}
\textbf{Condition \eqref{con4}$ \implies $}
\begin{align*}
\rho(y)_\m(m) + t\rho(x)_\l\b^{-1}(\rho(y)_\m m) = &\rho(y)_\m(m)+t \rho(y)_\mu(\rho(x)_\l\b^{-1}(m))+t\rho([x_\l\a^{-1}(y)])_{\l+\m}m\\&+  t^2 \rho([x_\l \a^{-1}(y)])_{\l+\m}(\rho(x)_\l \b^{-1}(m)). 
\end{align*}Comparing the coefficients of $t$ and $t^2$ in above equation, we get the following two equations
\begin{align*}
\rho(x)_\l\b^{-1}(\rho(y)_\m(m))= \rho(y)_\mu(\rho(x)_\l\b^{-1}(m))+\rho([x_\l\a^{-1}(y)])_{\l+\m}(m)
\end{align*}
and
\begin{align*}
\rho([x,\a^{-1}(y)])_{\l+\m}\rho(x)_\l(\b^{-1}(m))= 0, \text{ for all } y \in \L, m \in \M.
\end{align*}
By using the invertibility of $\a$ and $\b$, we get
\begin{align}\rho(x)_\l\rho(y)_\m(m) &= \rho[x_\l y]_{\l+\m}(m) + \rho(y)_\m\rho(x)_\l(m)),\\
\label{eq4con}\rho([x_\l y])_{\l+\m}\rho(x)_\l m&=0.\end{align}
\begin{thm}
	Let $\T : \M\to \L$ be an $\mathcal{O}$-operator. Let $\T^1_t := \T +t\mathfrak T_1$ and $\T^2_t := \T +t\mathfrak T_2$ be two
	equivalent linear deformations of $\T$. Then, $\mathfrak T_1$ and $\mathfrak T_2$ belong to the same cohomology class in $H^1_{\a,\b}(\M, \L)$.
\end{thm}
\begin{proof}The proof follows from equation \eqref{eqsaniaasif}. \end{proof}
\begin{defn} Let $\mathcal{T} : \mathcal{M} \rightarrow \mathcal{L}$ be an $\mathcal{O}$-operator on a Hom-Lie conformal algebra $(\mathcal{L}, [\cdot_{\lambda}\cdot],\alpha)$ with a representation $(\rho, \mathcal{M}, \beta)$. A linear deformation $\mathcal{T}_{t} := \mathcal{T} + t\mathfrak T$ is called trivial if it is equivalent to the deformation $\mathcal{T}_{0} = \mathcal{T}$.
\end{defn}
By the remark \eqref{rem2.1} and Eq. \eqref{differential}, we have the following modified cochain complex
$$(C^*_{\b,\a}(\M, \L), \delta_{\b,\a}).$$
 The set of modified $0$-cochain is defined by $C^0_{\b,\a}(\M,\L):= \{x\in \L| \a=id\},$ and the coboundary map $\delta_{\b,\a}: C^0_{\b,\a}(\L,\M)\to C^1_{\b,\a}(\L,\M)$ on $0$-cochain is given by $$ \delta_{\b,\a}(x)(m) := \rho_\T(\b^{-1}(m))_\l(x), $$ for all $x\in C^0_{\b,\a}(\M,\L)$, $m\in \M$  and  $\l\in \mathbb{C}$. According to Eq. \eqref{eqdifop}, we can extend the bracket $\{\!\!\{-,-\}\!\!\}$ to the modified cochain $C^*_{\b,\a}(\M, \L)$. For example,  
for any $x,y\in \L$, and $\T\in C^1_{\b,\a}(\M, \L)$ the bracket is given by 
$\{\!\!\{x,y\}\!\!\}=[x_\l y]$ and 
\begin{equation}
\begin{aligned}
\delta_\T(x)=\{\!\!\{\T,x\}\!\!\}(m)=\T (\rho(x)_\l(\b^{-1}(m)) )+ [ \T(\b^{-1}m)_{-\p-\l} x ] =\delta_{\b,a}(x)m.
\end{aligned}
\end{equation}

\begin{defn}Let $\T : \M \to \L$ be an $\mathcal{O}$-operator on the Hom-Lie conformal  algebra $(\L, [\cdot_\l\cdot],\a)$ with respect to a representation $(\rho,\mathcal{M},\beta)$. An element $x\in \L$ is called a Nijenhuis element associated to the operator $\T$ if $x $ satisfies $\a(x)=x$, the identities \eqref{eq2con}, \eqref{eq4con} and the identity $$[x_\m \T (\rho(x)_\l(m) )+ [\T(m)_{-\p-\l}x] ]=0, \text{ for all }m\in \M.$$ \end{defn} 
Let us denote the set of Nijenhuis elements associated to the $\mathcal{O}$-operator $\T$ by $Nij(\T)$.
\begin{thm}Let $\T : \M \to \L$ be an $\mathcal{O}$-operator on a Hom-Lie conformal algebra $(\L, [\cdot_\l\cdot],\a)$ with respect to a representation $(\rho,\M,\b)$. For any element $x\in Nij(\T)$, the linear deformation $\T_t:=\T + t\mathfrak{T}$ generated by $\mathfrak T:= \delta_\T (x)$ is a trivial deformation of $\T$. \end{thm}
\begin{proof} For the detailed proof, see \cite{Mishra-Naolekar, Tang-Guo-Sheng}.
\end{proof}
\subsection{ Formal deformations:}
Let $\mathcal{T}:\mathcal{M} \rightarrow \mathcal{L}$ be an $\mathcal{O}$-operator on a Hom-Lie conformal algebra $(\mathcal{L}, [\cdot_{\lambda}\cdot],\alpha)$ with respect to representation $(\rho,\mathcal{M},\beta)$. Let $\mathbb{C}[[t]]$ be the ring of  formal power series in one variable $t$ and $\mathcal{L}[[t]]$ be the formal
power series in $t$ with coefficients in $\mathcal{L}$. Then the triplet $(\L[[t]], [\cdot _{\lambda}\cdot ]_{t}, \alpha_{t})$ is a Hom-Lie conformal
algebra, where the bracket $[\cdot _\lambda \cdot ]_{t}$ and the structure map $\alpha_{t}$ are obtained by extending $[\cdot _\lambda \cdot ]$ and $\a$ linearly over the ring $\mathbb{C}[[t]]$. Moreover, the map $\rho : \mathcal{L} \rightarrow gc(\mathcal{M} )$ and $\beta : \mathcal{M} \rightarrow \mathcal{M}$ can be extended linearly over $\mathbb{C}[[t]]$ to obtain $\mathbb{C}[[t]]$-linear maps $\rho_{t} : \mathcal{L}[[t]] \otimes \mathcal{M} [[t]] \rightarrow \mathcal{M} [[t]]$ and $\beta_{t} : \mathcal{M} [[t]] \rightarrow \mathcal{M} [[t]]$. Then the triplet $(\mathcal{M} [[t]], \beta_{t}, \rho_{t})$ is a Hom-Lie conformal algebra representation of $(\mathcal{L}[[t]], [ \cdot_{\lambda} \cdot ]_{t}, \alpha_{t})$.
\begin{defn}
	Let $\mathcal{T} : \mathcal{M}\rightarrow \mathcal{L}$ be an $\mathcal{O}$-operator on a Hom-Lie conformal algebra $(\mathcal{L}, [\cdot_{\lambda}\cdot],\alpha)$ with a representation $(\mathcal{M},\beta,\rho)$. And let $\mathcal{T}_{t} : \mathcal{M} [[t]] \rightarrow \mathcal{L}[[t]]$ be an $\mathcal{O}$-operator on $(\mathcal{L}[[t]], [\cdot_{\lambda}\cdot]_{t},\alpha_{t})$ with the representation $(\mathcal{M} [[t]], \beta_{t}, \rho_{t})$. If $\mathcal{T}_{t}$ = $\mathcal{T}_{0} + \sum_{i=1}^{\infty} t^{i}\mathcal{T}_{i}$,
	where $\mathcal{T}_{i} \in C^{1}_{\b,\a}(\mathcal{M},\mathcal{L})$ and $\mathcal{T}_{0} = \mathcal{T}$, then $\mathcal{T}_{t}$ is called a formal deformation of $\mathcal{T}$ .
\end{defn}
Equivalently,
\begin{equation}\T_t(\b(m))=\b(\T_t(m)),\label{45}
\end{equation}
\begin{equation}\label{46}
[\T_t( m), \T_t(n)] = \T_t(\{\T_t (m),n\}- \{\T_t (n), m\}), \quad \forall m,n \in \M.
\end{equation}
Eq. \eqref{45} holds trivially, because $\T_i \in C^1_{\b,\a}(\M,\L)$. By comparing the coefficients of $t^k$ in Eq. \eqref{46}, we get the following  expressions:
\begin{equation}
\sum_{i+j=k}[\T_i(m), \T_j(n)] = \sum_{i+j=k} \T_i(\{\T_j(m),n\}- \{\T_j(m), n\}),\quad k= 0,1,..
\end{equation}
\begin{itemize}
	\item If $k=0$ we have, $\T_0=\T$. 
	\item If $k=1$, we have \begin{equation}
	[\T_1(m), \T_0(n)]+[\T_0(m), \T_1(n)] = \T_1(\{\T_0(m),n\}- \{\T_0(m), n\})+ T_0(\{\T_1(m),n\}- \{\T_1(m), n\}), 
	\end{equation}
	which is equivalent to the condition: $\delta_\T(\T_1)= 0$.
\end{itemize} Therefore, we get the proposition \eqref{prop3.8}.
\begin{prop}\label{prop3.8}
	The infinitesimal of the deformation $\mathcal{T}_{t}$ is a $1$-cocycle in the cohomology
	of the $\mathcal{O}$-operator $\mathcal{T}$. And the $k$-infinitesimal of the deformation $\mathcal{T}_{t}$ is a $1$-cocycle in the
	cohomology of the $\mathcal{O}$-operator $\mathcal{T}$ .
\end{prop}
\begin{defn}
	Let $\mathcal{T}_{1} \in C^{1}_{\b,\a}(\M,\L)$ be a $1$-cochain, it is said to be infinitesimal of the deformation $\mathcal{T}_{t}$.
\end{defn}
\begin{rem}
	If $\mathcal{T}_{i} = 0$, where $1 \leq i \leq k-1$ and $\mathcal{T}_{k}$ is a non-zero cochain, then $\mathcal{T}_{k}$ is said to be $k$-infinitesimal of the deformation $\mathcal{T}_{t}$.
\end{rem}
\begin{defn}
	Let $\mathcal{T}_{t}$ and $\mathcal{T}^{\prime}_{t}$ be two deformations of the $\mathcal{O}$-operator $\mathcal{T}$, then they are called equivalent if there exists an element $x \in C^{0}_{\b,\a}(\mathcal{M},\mathcal{L})$ and two $\mathbb{C}$-linear maps $\varphi_{i}:\mathcal{L} \rightarrow \mathcal{L}$ and $\psi_{i}: \mathcal{M} \rightarrow \mathcal{M}$, for $i \geq 2,$ such that 
	\begin{align*}
	\varphi_{t} = Id_{\mathcal{L}} + t(ad_{x}^{\dagger}) + \sum_{i=2}^{\infty} t^{i}\varphi_{i},\quad
	\psi_{t} = Id_\mathcal{M} + t\rho(x)^\dagger + \sum_{i=2}^{\infty} t^{i}\psi_{i}.\end{align*}
	Meanwhile, $(\varphi_{t}, \psi_{t})$ is a formal isomorphism from $\mathcal{T}_{t}$ to $T^{\prime}_{t}$. For $x \in C^{0}_{(\alpha, \beta)}(\mathcal{L}, \mathcal{M})$, the maps $ad_{x}^\dagger : \mathcal{L} \rightarrow \mathcal{L}$ and $\rho(x)^\dagger :\mathcal{M} \rightarrow \mathcal{M}$ are defined by $ad_{x}^\dagger(y) = \alpha^{-1}(ad_{x}(y))$ and $\rho(x)_{\l}^\dagger(m) = \beta^{-1}(\rho(x)_\l(m))$, for all $x, y \in \mathcal{L}~and~ m \in \mathcal{M}$.
\end{defn}
\begin{rem}\label{rem5.2}
	By the above definition, we obtain the following conclusion:
	\begin{enumerate}
		\item \label{codition1}$ \varphi_{t} \circ \alpha = \alpha\circ \varphi_{t}$ and $\psi_{t}\circ \beta = \beta\circ  \psi_{t}$,
		\item \label{codition2} $\rho_{t}(\varphi_{t}(y))_{\m}( \psi_{t}(m)) =  \psi_{t}(\rho(y)_\m(m))$,
		\item \label{codition3}$ \varphi_{t}[y_{\m}z] = [\varphi_{t}(y)_\m \varphi_{t}(z)]$,
		\item \label{codition4} $\varphi_{t} \circ \mathcal{T}_{t} = \mathcal{T}^{\prime}_{t}\circ\psi_{t}$.
	\end{enumerate}
\end{rem}
\begin{prop}
	Let two deformations $\mathcal{T}_{1}$ and $\mathcal{T}^{\prime}_{1}$
	of the $\mathcal{O}$-operator $\mathcal{T}$ on a Hom-Lie conformal algebra $(L, [\cdot_{\lambda}\cdot],\alpha)$ with a representation $( \rho,\mathcal{M},\beta)$ are equivalent. Then the infinitesimals of equivalent deformations belong to the same cohomology class in $\tilde{H}_{\b,\a}(\mathcal{M},\mathcal{L})$.
\end{prop}
\begin{proof}
By condition \eqref{codition4} of the Remark \eqref{rem5.2}, it follows that\begin{align*}\mathcal{T}_{1}(m)-\mathcal{T}^{\prime}_{1}(m)&
		= \mathcal{T} \rho(x)_\l\b^{-1}(m) + [\mathcal{T} (\b^{-1}(m))_{-\p-\l} x]\\&
		= \rho_{T} (\b^{-1}(m))(x)\\&= (\delta_{\b,\a}(x))(m).
	\end{align*}
		Consequently, $\T_{1}$ and $\T^{'}_{1}$ are in the same cohomology class $\tilde{H}_{\beta,\alpha}(\mathcal{M},\mathcal{L})$. 
\end{proof}
\subsection{ Obstructions in extending a finite order deformation to the next order:} Consider a Hom-Lie conformal algebra $(\L,[\cdot_\l\cdot],\a)$ and  the space $\L[[t]]/ t^{k+1}$, which is a $\mathbb{C}[[t]]/ t^{k+1}$-module. The structure map $\a_t$ and $\l-$bracket $[\cdot_\l\cdot]_t$ can be extended to $\L[[t]]/t^{k+1}$ by $\mathbb C[[t]]/t^{k+1}$-linearity making $(\L[[t]]/ t^{k+1}, [\cdot_\l\cdot]_t, \a_t)$ a Hom-Lie conformal algebra over $\mathbb{C}[[t]]/ t^{k+1}$.
\begin{defn} Consider  $\T: \M \to \L$ be an $\mathcal{O}$-operator on the Hom-Lie conformal algebra $(\L, [\cdot_\l\cdot], \a)$. An order $k$ deformation of the $\mathcal{O}$-operator $\T$ is a sum $$\T_t = \T_0+t\T_1+ t^2 \T_2+ \cdots \quad( modulo ~t^{k+1}),$$ where $\T_t$ is an $\mathcal{O}$ operator on the Hom-Lie conformal algebra $(\L[[t]]/ t^{k+1},[\cdot_\l\cdot]_t,\a_t)$ with respect to the representation $( \rho_{t}, \mathcal{M} [[t]], \beta_{t})$. Also $\T_0= \T$ and $\T_i\in C^1_{\b,\a}(\M,\L)$.\end{defn}
As we know that for $\T_t$ to be the deformation of order $k$, the following conditions hold:
\begin{align*}
\nonumber\a \circ \T_k =&~ \T_k \circ \b,\\
\label{eqmeri}
\sum_{i+j=k}[{\T_i(m)}_{\l} \T_j(n)]= &\sum_{i+j=k}\T_i(\{{\T_j(m)}_{\l} n\}+ \{m_{\l}{\T_{j}(n)}\}), \quad \forall ~m,n\in \M.
\end{align*}
for $i,j= 0, 1,\cdots, k$.
\begin{defn}A deformation $\T_t =\sum_{i=0}^{k}t^i\T_i$ of order $k$ is said to be extensible if there exists an element $\T_{k+1} \in C^1_{\b,\a}(\M,\L)$ with $\a \circ \T_{k+1}= \T_{k+1} \circ \b$ such that $\bar{\T_t }= \sum_{i=0}^{k+1}t^i\T_i$ defines a deformation of order $k+1$. \end{defn}
If $\T$ is extensible, then following deformation equation also needs to be satisfied: 
\begin{align*}
\sum_{i+j=k+1}[{\T_i(m)}_{\l} \T_j(n)]= \sum_{i+j=k+1}\T_i(\{{\T_j(m)}, n\}+ \{m, {\T_{j}(n)}\}), \quad \forall ~m, n\in \M.\end{align*}
Note that the above equation is equivalent to \begin{align*}
[\T(m)_{\l} \T_{k+1}(n)]+&[\T_{k+1}(m)_{\l} \T(n)]+\sum_{i+j=k+1}[\T_i(m)_{\l} \T_j(n)]\\=&
\T(\{\T_{k+1}(m), n\}+ \{m,\T_{k+1}(n)\})+\T_{k+1}(\{\T(m), n\}+ \{m, \T(n)\})\\&+\sum_{i+j=k+1} \T_i(\{\T_j(m), n\}+ \{m, \T_j(n)\}).\end{align*}
\begin{align*}
\sum_{i+j=k+1}([\T_i(m)_{\l} \T_j(n)]&- \T_i(\{\T_j(m), n\}+ \{m, \T_j(n)\}))\\=&-[\T(m)_{\l} \T_{k+1}(n)]-[\T_{k+1}(m)_{\l} \T(n)]+ 
\T(\{\T_{k+1}(m), n\}+ \{m,\T_{k+1}(n)\})\\&+\T_{k+1}(\{\T(m), n\}+ \{m, \T(n)\})\\=&-\delta_{\T}(\T_{k+1})(m,n).\end{align*}
Above equation can be written as
\begin{equation}\label{aiza}\delta_{\T}(\T_{k+1})= -\frac{1}{2}\sum_{i+j=k+1,i,j>1}\{\!\!\{\T_i , \T_j\}\!\!\}.
\end{equation}The right-hand side of the above equation does not contain $\T_{k+1}$. It is known as the obstruction to extend the deformation $\T_t$. We denote the obstruction by $Ob_{\T_t}$.

\begin{prop}\label{eqtum}
	In the cohomology of $\mathcal{O}$-operator, the obstruction is a $2$-cocycle, i.e., $\delta_{\T} (Ob_{\T_t} )=~0$.
\end{prop}
\begin{proof}We have\begin{equation*}\begin{aligned}d_{\T}(Ob_{\T_t})&= -\frac{1}{2}\sum_{i+j=k+1,i,j>1}\{\!\!\{\T, \{\!\!\{\T_{i} , \T_{j}\}\!\!\}\}\!\!\}\\&= -\frac{1}{2}\sum_{i+j=k+1,i,j>1}(\{\!\!\{\{\!\!\{ \T , \T_{i}\}\!\!\}, \T_{j}\}\!\!\} -\{\!\!\{\T_i, \{\!\!\{\T , \T_{j}\}\!\!\}\}\!\!\})\\&=\frac{1}{4}\sum_{i_1+i_2+j= n+1, i_1,i_2,j>1}\{\!\!\{\{\!\!\{ \T_{i1} , \T_{i2}\}\!\!\}, \T_{j}\}\!\!\}- \frac{1}{4}\sum_{i+j_1+j_2= k+ 1,i,j_1,j_2>1}\{\!\!\{\T_i, \{\!\!\{\T_{j1} , \T_{j2}\}\!\!\}\}\!\!\}\\&= \frac{1}{2}\sum_{i+j= k+ 1,i,j>1}\{\!\!\{\{\!\!\{ \T_{i} , \T_{j}\}\!\!\}, \T_{k}\}\!\!\}\\&=0.\end{aligned}\end{equation*}
\end{proof}Thus, from the Proposition \eqref{eqtum} and Eq. \eqref{aiza}, the following theorem applies.\begin{thm}Let $\T_t$ be be an order $k$ deformation of an $\mathcal{O}$-operator $\T$. This deformation $\T_t$ is extensible if and only if the obstruction class $[Ob_{\T_t}]\in \tilde{H}^2_{\T}(\M,\L)$ is trivial, where we denote $\tilde{H}^2_{\T}(\M,\L)$ as a cohomology class associated with the $2$-cochain complex.
\end{thm}
\begin{cor}
	Any finite order deformation of $\T$ is extensible, if $\tilde{H}^2_{\T}(\M,\L)= 0$.
\end{cor}
\section*{Data availability}All data is available within the manuscript.
\section*{Declarations}\textbf{Conflict of Interests:} The authors have no conflicts of interest to declare that are relevant to the content of this article.\\
\noindent {\bf Acknowledgment:}
The authors would like to thank the referee for valuable comments and constructive suggestions on this article.

\end{document}